\documentclass{amsart}

\usepackage{lipsum}
\usepackage{enumitem}
\usepackage{cleveref}
\usepackage{amsfonts}
\usepackage{amssymb}
\usepackage{amsmath}
\usepackage{amsthm}
\usepackage{colortbl}
\usepackage{mathtools}
\usepackage{graphicx,xcolor}
\usepackage{epstopdf}
\usepackage{booktabs}
\usepackage{caption}
\usepackage{subcaption}
\usepackage{algorithmic}
\usepackage[left=1.25in,right=1.25in]{geometry}
\ifpdf
  \DeclareGraphicsExtensions{.eps,.pdf,.png,.jpg}
\else
  \DeclareGraphicsExtensions{.eps}
\fi

\newtheorem{lemma}[section]{Lemma}
\newtheorem{theorem}[section]{Theorem}

\crefname{hypothesis}{Hypothesis}{Hypotheses}

\newtheorem{exm}[subsection]{Example}
\crefname{exm}{Example}{Examples}

\title{Conformal Finite Element Methods for nonlinear Rosenau-Burgers-biharmonic models}\thanks{
Ankur has been supported by the University Grant Commission India through ID. No.  JUNE18-416131; Ram Jiwari has been supported by the National Board of Higher Mathematics (NBHM), India through grant No. 02011/3/2021NBHM(R.P)/R\&D II/6974;  A. Narayan was partially supported by NSF DMS-1848508.  This material is based upon work supported by both the National Science Foundation under Grant No. DMS-1439786 and the Simons Foundation Institute Grant Award ID 507536 while A. Narayan was in residence at the Institute for Computational and Experimental Research in Mathematics in Providence, RI, during the Spring 2020 semester.}

\author{Ankur\thankssymb{1}}\thanks{\thankssymb{1}Department of Mathematics, Indian Institute of Technology Roorkee, Roorkee-247667, India.
  Email: ankur@ma.iitr.ac.in.}
\author{Ram Jiwari\thankssymb{2}}\thanks{\thankssymb{2}Department of Mathematics, Indian Institute of Technology Roorkee, Roorkee-247667, India.
  Email: ram.jiwari@ma.iitr.ac.in.}
\author{Akil Narayan\thankssymb{3}}\thanks{\thankssymb{3}Department of Mathematics and Scientific Computing and Imaging (SCI) Institute, University of Utah, Salt Lake City, UT 84112, USA. Email: akil@sci.utah.edu.}

\usepackage{amsopn}

\def\R{{\mathbb R}}
\def\N{{\mathbb N}}
\newcommand{\pfpx}[2]{\frac{\partial #1}{\partial #2}}
\definecolor{Gray}{gray}{0.85}
\newcolumntype{a}{>{\columncolor{Gray}}c}
\makeatletter
\DeclareRobustCommand*\cal{\@fontswitch\relax\mathcal}
\makeatother
\makeatletter
\newcommand\thankssymb[1]{\textsuperscript{\@fnsymbol{#1}}}
\makeatother

\begin{document}

\maketitle

\begin{abstract}
We present a novel and comparative analysis of finite element discretizations for a nonlinear Rosenau–Burgers model including a biharmonic term. We analyze both continuous and mixed finite element approaches, providing stability, existence, and uniqueness statements of the corresponding variational methods. We also obtain optimal error estimates of the semidiscrete scheme in corresponding B\^{o}chner spaces. Finally, we construct a fully discrete scheme through a backward Euler discretization of the time derivative, and prove well-posedness statements for this fully discrete scheme. Our findings show that the mixed approach removes some theoretical impediments to analysis and is numerically easier to implement. We provide numerical simulations for the mixed formulation approach using $C^0$ Taylor-Hood finite elements on several domains. Our numerical results confirm that the algorithm has optimal convergence in accordance with the observed theoretical results. 
\end{abstract}

%

\section{Introduction}
The Rosenau–Burgers model is a special type of nonlinear partial differential equation (PDE) containing a biharmonic term that is of great importance in applications in condensed matter physics, nonlinear optics and fluid mechanics. The model is obtained by including a dissipative term in the Rosenau equation introduced by Roseanu \cite{rosenau1986quasi}. The Rosenau-type equations are widely used to describe wave-wave and wave-wall interactions. The Rosenau–Burgers model was first introduced in 1989 to describe the propagation of waves in shallow water, and also appears in natural phenomena such as water waves and bore propagation \cite{peregrine1966calculations}:
\begin{subequations}\label{eq:rn}
	\begin{eqnarray}
&u_{t}+\Delta^2u_t - \alpha \Delta u-\nabla\cdot g(u) =0\; \ \mbox{in}\;  \Omega\times(0,T],  \label{eq1*}
\end{eqnarray}
for some $T > 0$, supplemented with initial and homogeneous Dirichlet boundary conditions,
\begin{eqnarray}
&&u(x,t)=0\;=\;\Delta u(x,t)\;\;\mbox{for}\; x \in \partial \Omega,\label{eq2}\\
&&u(x,0) = u_{0}(x),\;\mbox{for}\;x\in\Omega, \label{eq3}
\end{eqnarray}
\end{subequations}
where $\Omega\subseteq\mathbb{R}^{n}$, $n \leq 2$ is a bounded domain and $g: \R \rightarrow \R^n$ is the vector function, 
 \begin{align*}
   g &= (g_{1}, g_{2},\ldots, g_{n})^T \in \R^n, & g_{i}(u)=-\Big(u+ \frac{1}{2}u^{2}\Big), \;\;\;\; \forall \; i \in [n] = \{1, \ldots, n\},
 \end{align*}
 where $[n] = \{1, 2, \ldots, n\}$.
 The dissipative coefficient $\alpha$ satisfies $\alpha > 0$. In addition, we assume that the initial data $u_0$ is smooth; we will articulate precise smoothness requirements later. 
\subsection{A brief literature review}
Constructing and implementing appropriate variational formulations for PDEs with biharmonic terms can be challenging. 
An alternative approach is to utilize nonconforming and mixed Galerkin finite element methods,
for example using piecewise quadratic polynomials on Morley elements for biharmonic problems \cite{morley1968triangular}, although there are several other classical examples for addressing such equations \cite{falk1978approximation, monk1987mixed, ciarlet1974mixed} and some recent advances, such as the $C^0$ linear FEM approach in \cite{guo2018ac}.
Similar approaches have been employed for problems of Rosenau-type as well \cite{cui2007numerical,zurnaci2019convergence,li2016numerical,rey2013large}.
For transport-reaction equations similar to Rosenau-type equations, $L^2$ error estimates have been derived \cite{cockburn2008optimal, cockburn2010optimal, zhang2020weak}.
Hybridizable discontinuous Galerkin methods also have \textit{a posteriori} error estimates for convection-diffusion equations \cite{chen2016robust}.
There is analysis for the existence, uniqueness, stability, and convergence of schemes for dispersive shallow water waves modeled by combining the Rosenau-Burgers equation and Rosenau-RLW equations \cite{wongsaijai2021optimal}.
For two-dimensional Rosenau-Burgers models, 
Fourier transform methods have been employed to understand convergence rates and long-time behavior \cite{liu2002better}. 
Finite element analysis for the type of model we consider has been carried out for second-order schemes using linearized algorithms \cite{rouatbi_efficient_2022,omrani_numerical_2022}, using splitting techniques \cite{atouani_mixed_2018}; both these approaches suffer from relatively low convergence rates.

\subsection{Contribution}
The main contributions of this article are in the numerical analysis and resulting implementation of implicit time-differencing finite element schemes for the Rosenau-Burgers-biharmonic problem \eqref{eq:rn}. In particular:
\begin{itemize}[itemsep=-2pt, topsep=0pt, leftmargin=20pt]
  \item We provide formulations for both primal and mixed variational formulations for \eqref{eq:rn}. See \eqref{eq:P1} and \eqref{eq:P2} for primal and mixed forms, respectively. We provide semidiscrete finite element formulations for these problems, see \eqref{eq77} and \eqref{eq:P2-semi}, respectively. Finally, we provide time-differencing fully discrete (directly implementable) schemes, see \eqref{eq85} and \eqref{eq:P2-full}, respectively.
  \item We provide the rigorous theory of the continuous-level, semidiscrete, and fully discrete schemes using analysis in B{\^o}chner spaces. For the primal formulation, our analysis contributions are:
    \begin{itemize}[topsep=-3pt,leftmargin=10pt]
      \item Existence and uniqueness of variational problem, see \cref{thm:P1-wellposed}
      \item Convergence rates for a finite element semidiscrete version of the problem, see \cref{thmerror}
      \item Well-posedness of an implicit fully discrete scheme, see \cref{thm:P1-full-wellposed}
      \item Convergence rates for the fully discrete scheme, see \cref{thmBE}
    \end{itemize}
  \item For the mixed formulation, our analysis contributions are:
    \begin{itemize}[topsep=-3pt,leftmargin=10pt]
      \item Existence and uniqueness of variational problem, see \cref{exist1}
      \item Convergence rates for a finite element semidiscrete version of the problem, see \cref{thmerror.}
      \item Well-posedness of an implicit fully discrete scheme, see \cref{thm:P2-full-wellposed}
      \item Convergence rates for the fully discrete scheme, see \cref{thmBE*}.
    \end{itemize}
  \item We provide numerical results for the mixed formulation in two spatial dimensions that confirm the scheme order of accuracy provided by our theory, see \cref{Num}
\end{itemize}
Of particular note are the following aspects of our theory: This paper provides the first attempt to derive a new type of convergence results in different B\^{o}chner spaces for a time-dependent problem resulting from analysis of both the primal and mixed variational formulations. We show that the numerical solutions of the proposed methods converge to the exact solution with optimal orders under $L^2$, $H^1$ and $H^2$ norms. Our results require less regularity on the initial condition for the mixed formulation compared to the primal formulation.

\subsection{Structure of the paper}
   The remaining part of this paper is organized as follows:
        \Cref{discretization} describes the weak formulations, the finite element discretization, and summarizes notation that will be used throughout the article. \Cref{Sec P1} contains our theory for the primal formulation, including analysis for the continuous-level problem (\cref{ssec:P1-exist}), the semidiscrete problem (\cref{ssec:P1-semi}), and the fully discrete problem (\cref{ssec:P1-fully}). \Cref{Sec P2} contains our theory for the mixed formulation, including analysis for the continuous-level problem (\cref{existmixed}), the semidiscrete problem (\cref{semimix}), and the fully discrete problem (\cref{fully-mix}). Finally, \cref{Num} contains numerical experiments in the one and two-dimensional cases using Taylor-Hood finite elements on several domains.

  \section{Problem setup and weak formulations}\label{discretization}
  \vspace{-0.2cm}
  We first introduce some necessary function spaces for our analysis.
  \subsection{Notation and assumptions}\label{ssec:notation}
  For any $n \in \N$, we write $[n] = \{1, 2, \ldots, n\}$. Let $\Omega \subset \R^n, n\leq 2$ be a bounded, Lipschitz domain. We define standard $L^2$-type Sobolev spaces and norms: Given $m = 0, 1, 2$,
  \begin{align*}
    H^m(\Omega) &\coloneqq \left\{ u: \Omega \rightarrow \R\;\;\big|\;\; \| u\|_{m} < \infty \right\}, & 
    \| u\|^2_{m} \coloneqq \left\{\begin{array}{rl} 
      \|u\|^2, & m = 0 \\
      \|u\|^2 + \|\nabla u \|^2, & m = 1 \\
      \|u\|^2 + \|\nabla u \|^2 + \|\Delta u \|^2, & m = 2
    \end{array}\right.
  \end{align*}
  where $\|\cdot\|$ denotes the standard $L^2(\Omega)$ norm. When $m = 0$, we have $H^0(\Omega) = L^2(\Omega)$, and we will use $(\cdot,\cdot)$ and $\|\cdot\|$ to denote the associated $L^2(\mathcal{V})$ inner product and norm, respectively.
  We require vanishing trace Sobolev spaces as well:
  \begin{align*}
    H_0^1(\Omega) &\coloneqq \left\{ u \in H^1(\Omega) \; \big|\; u|_{\partial \Omega} = 0 \right\}, & 
    H_0^2(\Omega) &\coloneqq \left\{ u \in H^2(\Omega) \; \big|\; u|_{\partial \Omega} = 0,\; \pfpx{u}{n}|_{\partial \Omega} = 0\right\},
  \end{align*}
  where $n$ denotes the outward pointing unit vector to $\partial \Omega$.  We will leverage the fact that $u \mapsto \|\nabla u \|$ is a norm on $H_0^1(\Omega)$. 

  We will perform analysis on time-dependent functions in B{\^o}chner spaces. For a terminal time $T > 0$, let $I = (0, T)$. For $\mathcal{Z}$ a Banach space, $p \in \{1, 2, \infty\}$, we define norms on $L^p(I;\mathcal{Z})$ as,
  $$\|w\|_{L^p(I; {\cal Z})}=\left\{
  \begin{array}{ll}
  \int_0^T\|w(t)\|_{\cal Z}dt, & \mbox{when}\;p=1, \\
  \Big(\int_0^T\|w(t)\|_{\cal Z}^2dt\Big)^{\frac{1}{2}}, & \mbox{when}\;p=2, \\
  \mbox{ess}\sup_{t\in [0,T]}\|w(t)\|_{\cal Z}, & \mbox{when}\;p=\infty.
  \end{array}
  \right.$$ 
  We will use ${\partial_{t}u}$, $u_t$ or $\dot{u}$ interchangeably to denote differentiation of $u$ with respect to the temporal variable $t$. 
  In our analysis below, we let $C,K,C_{i}$ and $K_{i}$, for $i\in \mathbb{N}$ denote constants that are independent of discretization parameters, but may depend on $\Omega$ and $T$. 

  For finite element discretizations, we assume that $\Omega$ is polyhedral, and let $\mathcal{T}_h$ denote a conformal partition of $\Omega$ into simplices with controllable mesh quality, i.e.,:
  \begin{itemize}[itemsep=0pt,topsep=0pt,leftmargin=25pt]
  	\item  ${\overline \Omega = \cup_{J\in{\cal T}_h} J}$.
  	\item If $ J_1 $, $ J_2 $ $\in{\cal T}_h $ and $
  	J_1 \ne J_2 $, then either $J_1 \cap J_2  = \emptyset $ or $J_1 \cap J_2 $ is a common $m$-face of both simplices, for $0 \leq m \leq n-1$. Hence, for $n=2$ a triangle's vertex interior to $\Omega$ can only intersect other triangles at their vertices.
      \item Let $h_J$ represent the largest $1$-face length (edge length) of the simplex $J\in {\cal T}_{h}$ and suppose $h = \max_{J\in{\cal T}_{h}} h_{J}$ is the maximum size of a side of a triangulation ${\cal T}_{h}$. We assume the following about triangulation $\mathcal{T}_h$ under consideration: for any fixed $h' > 0$, there always are positive constants $K_0$ and $K_1$, such that for any triangulation $\mathcal{T}_h$ with $0 < h < h'$, we have
  	\begin{eqnarray*}
  		K_0 h  \leq \mbox{diam}(J) \leq K_1h~~\forall\; J \in{\cal T}_h,
  	\end{eqnarray*}
        where $K_0$ and $K_1$ are independent of $h$.
  \end{itemize}

\subsection{The weak formulations}\label{ssec:weak}
We present the two weak formulations that we focus on in the remainder of the document.  

\paragraph{Problem $P^1$:}
Using Green's formula and  taking the $L^2$- inner product of equation \eqref{eq1*} with $\chi \in H_0^2(\Omega)$, we identify a variational formulation for \eqref{eq1*}. Problem $P^1$ is as follows: We seek a function $u\in L^{\infty}(I;H_{0}^{2}(\Omega)) \cap L^{2}(I;H_{0}^{1}(\Omega))$ with $u_t \in L^{2}(I;H^{2}(\Omega))$ that satisfies
\begin{subequations}\label{eq:P1}
\begin{align}
(u_{t}(t),\chi)+(\Delta u_{t}(t),\Delta \chi)+\alpha(\nabla u(t), \nabla \chi)&=(\nabla\cdot g(u),  \chi), \quad\forall \chi \;\in H_{0}^{2}(\Omega),\label{eq4a}\\
u(x,0)&=u_{0}(x), \quad\hspace{0.76cm} \forall\; x \;\in \Omega.\label{eq4b}
\end{align}\label{eq4}
\end{subequations}

\paragraph{Problem $P^2$:}
As done in \cite{burman2022hybrid}, introduce $p=-\Delta u$ to obtain a mixed weak formulation of the fourth-order Rosenau-Burgers model \eqref{eq1*}. With the variable $p$, \eqref{eq1*} becomes
\begin{align}
u_t-\Delta p_t+\alpha p-\nabla\cdot g(u)&=0\nonumber,\\  p+\Delta u&=0,\nonumber\\
u(x,t)=0\;=\; p(x,t)\;\;\mbox{for}\;& x \in \partial \Omega,\nonumber\\
u(x,0) = u_{0}(x),\;\mbox{for}\;&x\in\Omega.\nonumber
\end{align}
This allows us to articulate Problem $P^2$: We seek a function pair $(u,p)$ satsifying $u\in L^{\infty}(I;L^{2}(\Omega)) \cap L^{2}(I;H_{0}^{1}(\Omega))$ with $u_t\in L^{2}(I;L^{2}(\Omega))$, and $p \in L^{\infty}(I;\Xi_0)$ with $p_t\in L^{2}(I;H^{1}(\Omega))$, that satisfies
\begin{subequations}\label{eq:P2}
\begin{align}
(u_{t}(t),\chi)+(\nabla p_{t}(t),\nabla \chi)+\alpha(p(t),  \chi)&=(\nabla\cdot g(u),  \chi), \quad\forall \;\chi \;\in H_{0}^{1}(\Omega), \label{eq5a}\\
(\nabla u(t),\nabla \chi')&=(p(t),  \chi'), \quad\;\hspace{0.12cm}\forall\; \chi' \;\in H_{0}^{1}(\Omega),\label{eq5b} \\
u(x,0)&=u_{0}(x),\quad\hspace{0.74cm} \forall\; x \;\in \Omega. \label{eq5c}
\end{align}\label{eq5}
\end{subequations}
\noindent Where, $\Xi_0=\{u\in L^2(\Omega)\;:\;u|_{\partial\Omega}=0\}$; note that $\Xi_0$ requires functions $u$ to be regular enough to have traces on $\partial \Omega$. The mixed formulation $P^2$ is useful to circumvent the theoretical and computational challenges that arise in working with $H^2$ functions.

Next, we present the following lemma used throughout the article. 
\begin{lemma}\label{lemma:2.2}
  The function $u \mapsto \|u\| + \|\Delta u\|$ is a norm on $H^2(\Omega)$ and is equivalent to $\|u\|_2$.
\end{lemma}
We provide the proof in \Cref{sec:lemma2.2-proof}.

\section{Analysis of problem $P^1$} \label{Sec P1}
In this section, we focus on Problem $P^1$ from \cref{ssec:weak}. Our first set of results in \cref{ssec:P1-exist} is well-posedness statements for the variational problem \eqref{eq:P1}. We then consider a finite element discretization of the problem and prove corresponding spatial convergence rates in \cref{ssec:P1-semi}. Finally, we discretize the time variable and provide well-posedness statements and convergence rates for a fully discrete scheme in \cref{ssec:P1-fully}.
\subsection{Well-posedness of the variational problem}\label{ssec:P1-exist}
\vspace{-0.2cm}
In this section, we shall prove the existence and uniqueness of the weak solution to equation (\ref{eq4}). Before we prove the existence and uniqueness, it is generally beneficial to know the long-term behavior of the solution to a time-dependent equation. In particular, one would like to know if the solution grows over time or if it can be constrained by given initial data. Stability estimates are very useful for this purpose.
\begin{lemma}[Stability estimate]\label{thm1}
	Let $u$ be a solution of \eqref{eq4} and suppose $u_{0} \in H_{0}^{2}(\Omega)$. Then the following property holds:
        \begin{equation}\label{eq:P1-stab-1}
	\|u(t)\|_{2}  \leq  C\|u_{0}\|_{2}  ,  \;\;t \in(0,T].
	\end{equation} 
	Moreover, there exists a constant $C_1 > 0$ such that
        \begin{equation}\label{eq:P1-stab-2}
	\|u\|_{L^{\infty}(I;L^{\infty}(\Omega))}  \leq  C_1\|u_{0}\|_{2}.  
	\end{equation}
\end{lemma}
\begin{proof}
Setting $\chi = u$ in the equation \eqref{eq4}, we have
\begin{equation}	
\dfrac{1}{2} \dfrac{\mathrm{d}}{\mathrm{d}t}\Big[\|u\|^{2} + \|\Delta u\|^{2}\Big]  +   \alpha \|\nabla u\|^{2}  = \int_{\Omega} (\nabla\cdot g(u))ud\Omega .\label{eqn5*}
\end{equation}
Now, let us define 
$ G(s):= \int_{0}^{s} g(\omega)d\omega$  with $G^{'}(s) =g(s),\;\; s \in \mathbb{R}$
such that
\begin{equation}
\nabla\cdot G[u(x,t)]= g[u(x,t)]\cdot\nabla u(x,t). \label{eqn6}
\end{equation}
Now, by applying the divergence theorem \cite{quarteroni2009numerical} with $G(0)=0$ and $u=0$ on $\partial\Omega$, we get 
\begin{equation}
\int_{\Omega} (\nabla\cdot g(u))u d\Omega = - \int_{\Omega} g(u)\cdot \nabla u d\Omega = - \int_{\Omega} \nabla\cdot G(u)d\Omega = 0. \label{eqn7}
\end{equation}
Using equations \eqref{eqn7} in \eqref{eqn5*} and noting that $\alpha > 0$, we obtain 
$$ \dfrac{\mathrm{d}}{\mathrm{d}t}\Big[\|u\|^{2} + \|\Delta u\|^{2}\Big]\leq 0 . $$
Integrating the above with respect to time $t$ results in
$$\|u\|^{2} + \|\Delta u\|^{2} \leq \|u_{0}\|_{2}^2.$$
Combining this with \cref{lemma:2.2} implies
\begin{equation}
\|u(t)\|_{2}  \leq  C\|u_{0}\|_{2}  ,  \;\;t \in(0,T],
\end{equation} 
which proves \eqref{eq:P1-stab-1}. Finally, using the fact that $H^{2}(\Omega)$ is embedded in $L^{\infty}(\Omega)$ (cf. Sobolev’s Inequality, \cite{quarteroni2009numerical}) in \eqref{eq:P1-stab-1} proves \eqref{eq:P1-stab-2}.
\end{proof}

\begin{theorem}[Problem $P^1$ existence and uniqueness]\label{thm:P1-wellposed} Assume that $u_{0} \in H_{0}^{2}(\Omega)$. Then for $T>0$, there exists a unique weak solution to the problem (\ref{eq1*}) in the sense of the weak formulation \eqref{eq4}.
\end{theorem}
\begin{proof}
We shall prove the theorem in the following six parts:
\begin{enumerate}[leftmargin=13pt]
  \item \textit{Compression of the weak problem:} Let  $\{\phi_{j}\}^{\infty}_{j=1}$ be any orthonormal basis of $H_{0}^{2}(\Omega)$ and define $\digamma^{m}$ = span$\{\phi_{1},\phi_{2},\ldots,\phi_{m}\}$ for any $m \in \N$. Fixing $m$, we define the following compression of \eqref{eq:P1}: Let $u^m: [0,T] \rightarrow \digamma^m$ be given by
	\begin{equation}
	u^{m}(t) = \sum_{j=1}^{m} d_m^j(t)\phi_{j}, \label{eq30}
	\end{equation}	
        where an evolution for the $d_m^j$ is prescribed by a finite-dimensional version of \eqref{eq:P1}:
	\begin{align}
	(u_{t}^m(t),\phi_j)+(\Delta u_{t}^m(t),\Delta \phi_j)+\alpha(\nabla u^m(t), \nabla \phi_j)=(\nabla\cdot g(u^m),  \phi_j),  \label{eq31}
	\end{align}
	with $(u^m(0), \phi_j)=(u_{0}, \phi_j),\;\;\;\mbox{for} \;\phi_j \in \digamma^{m}, j=1,2,\ldots, m.$ 
        \eqref{eq31} is a nonlinear system of ordinary differential equations, and the Picard existence theorem guarantees that there exists a $T^*$ with $0<T^*<T$ such that $u^m(t) \in \digamma^m$ for $t\in [0,T^*].$ Our next goal is to extend the time $T^*$ to $T$.
      \item \textit{$H^2$ stability estimate:} Multiplying \eqref{eq31} by $d_m^j(t)$ and then summing over $j \in [m]$ yields
	\begin{align}
          (u_{t}^m(t),u^m(t))+(\Delta u_{t}^m(t),\Delta u^m(t))+\alpha(\nabla u^m(t), \nabla u^m(t)) &= (\nabla\cdot g(u^m),  u^m(t)),  \nonumber\\
          &\Downarrow\nonumber\\
          \frac{1}{2}\frac{\mathrm{d}}{\mathrm{d}t}\big[\|u^m(t)\|^2+\|\Delta u^m(t)\|^2\big] +\alpha \|\nabla u^m(t)\|^2&=(\nabla\cdot g(u^m),  u^m(t)).\label{eq32}
	\end{align}

	Integrating \eqref{eq32} from $0$ to $t$ and using \eqref{eqn7} yields
	\begin{equation}
	\|u^m(t)\|^2+\|\Delta u^m(t)\|^2 +\alpha\int_{0}^{t} \|\nabla u^m(s)\|^2ds \leq\|u_0\|_{2}^2, \nonumber
	\end{equation}
	where we have used the fact that $\|u^m(0)\|_2 \leq \|u_0\|_2$. 
	Finally, using \cref{lemma:2.2} and taking the supremum over $0\leq t\leq T$ yields
	\begin{equation}
	\sup_{0\leq t\leq T}\|u^m(t)\|_{2}^2 +\alpha\int_{0}^{T} \|\nabla u^m(t)\|^2dt\leq\|u_0\|_{2}^2.  \label{eq37}
	\end{equation}
        We emphasize that the right-hand side above is independent of $m$.
	\item \textit{Estimate of time derivative:}
	From \eqref{eq37}, we have
        \begin{equation*}
	\|u^m\|_{L^{\infty}(I;H_{0}^{2}(\Omega))}\leq C_1\,\{\|u_0\|_{2}\},
        \end{equation*}
	where $C_1$ is a positive constant depending on $\|u_0\|_{2}$. 
	From the above, we have the following
	\begin{equation}
	\|\nabla u^{m}\|_{L^{\infty}(I;L^{2}(\Omega))} \leq \| u^{m}\|_{L^{\infty}(I;H_{0}^{2}(\Omega))} \leq C_1 \{\|u_0\|_{2}\}. \label{eq38}
	\end{equation}
	Also, we have 
          {\small
	\begin{align}
	\|g(u^{m})\|^{2}=2\int_{\Omega} \Big(\frac{1}{2}(u^{m})^{2} + u^{m} \Big)^{2} d x 
          &\leq 2\bigg( \frac{1}{2} \int_{\Omega} (u^{m})^{4} d x +2\int_{\Omega} (u^{m})^{2} d x \bigg)
	\leq C_2 \{\|u_0\|_{2}\}, \label{eq39}
	\end{align}}
        where the last inequality uses the 
        result $\|u^{m}(t)\|_{L^{p}(\Omega)} \leq C_{3}\|u^{m}(t)\|_{H^{2}(\Omega)}$ for $\;1 \leq p\leq\infty$ \cite{quarteroni2009numerical}. Note that $C_2$ is a positive constant depending on $\|u_0\|_{2}$. Consequently,  $g(u^{m})$ is bounded in  $L^{\infty}(I;L^{2}(\Omega)).$
	
        To estimate ${u_t^m}$, we multiply \eqref{eq31} by $d_{m,t}^j(t)$ (the $t$ derivative of $d_m^j(t)$) and then sum over all $j \in [m]$, obtaining
	\begin{align}
          (u_{t}^m(t),u_{t}^m(t))+(\Delta u_{t}^m(t),\Delta u_{t}^m(t))&=-\alpha(\nabla u^m(t), \nabla u_{t}^m(t)) -( g(u^m), \nabla u_{t}^m(t)). \label{eq83}
	\end{align}
	Using Young's inequality and Cauchy-Schwarz inequality on the right-hand side above, and using \cref{lemma:2.2} on the left-hand side, there is a constant $C_4$ such that
	$$\|u^{m}_t(t)\|_2^2\leq C_4(\|g(u^{m})\|^{2}+\|\nabla u^{m}(t)\|^2 ).$$
        Integrating the above equation from $0$ to $T$ and using the bounds \eqref{eq38}-\eqref{eq39}, there is a positive constant $C_5$ depending on $\|u_0\|_{2}$ such that
	\begin{equation}
	\int_{0}^{T}\|u^{m}_t(t)\|_2^2dt\leq C_5\{\|u_0\|_{2}\} .\label{eq40}
	\end{equation}
	\item \textit{Passing to the limit:} The estimates \eqref{eq37} and \eqref{eq40} show that certain norms of $u^m$ are independent of $m$, and hence the sequence $\{u^m\}_{m=1}^{\infty}$ is uniformly bounded. Consequently, by the Banach-Alaoglu theorem \cite{chung1994finite,brezis2011functional}, one can extract a subsequence $\{u^{m_k}\}_{k=1}^{\infty}$ of $\{u^m\}_{m=1}^{\infty}$ such that 
	\begin{align}
          u^{m_{k}} &\xrightarrow{\textrm{weak}-\ast}
          u \;\;\textrm{in}\;\; L^{\infty}(I; H_{0}^{2}(\Omega)), \nonumber\\
          u^{m_{k}} &\xrightarrow{\textrm{weakly}}
          u \;\;\textrm{in}\;\; L^{2}(I; H_{0}^{1}(\Omega)), \nonumber\\
          u^{m_{k}}_t &\xrightarrow{\textrm{weakly}}
          u_t \;\;\textrm{in}\;\; L^{2}(I; H^{2}(\Omega)),\label{eq41}	
	\end{align}
        for some function $u$.
	In addition, due to the Rellich–Kondrachov theorem, the embedding of $H_{0}^{2}(\Omega)$ is compact in $L^{2}(\Omega)$ \cite{evans2022partial, taylor1996partial}, hence we can use the Aubin-Lions compactness lemma \cite{simon1986compact} to assert the following strong convergence
	\begin{eqnarray}
          &&u^{m_{k}} \longrightarrow \;u \;\;\textrm{in}\;\; L^{2}(I; L^{2}(\Omega)).\label{eq41*}
	\end{eqnarray}
        We now consider the subsequence of problems \eqref{eq41} for $m \gets m_k$, and take $k \uparrow \infty$. To accomplish this, 
define a function $\chi\in C^1([0,T]; H_0^2(\Omega))$ with the expression $\chi(t) = \sum_{j=1}^{M} f_m^j(t)\phi_{j}$, where $\{f_m^j(t)\}_{j=1}^{M}$ are smooth functions. Now, choosing $m\geq M$, multiplying equation \eqref{eq31} by $f_m^j(t)$, summing over $j \in [M]$,  and then integrating it from $0$ to $T$, we obtain 
{\footnotesize
	\begin{align}
	\int_{0}^{T}\bigg[(u_{t}^m(t),\chi(t))+(\Delta u_{t}^m(t),\Delta \chi(t))+\alpha(\nabla u^m(t), \nabla \chi(t))\bigg]dt
         & =\int_{0}^{T}\bigg[(\nabla \cdot g(u^m),  \chi(t))\bigg]dt.  \label{eq42}
	\end{align}
        }
	We now choose $m=m_k$ and use the convergence statements in \eqref{eq41} and \eqref{eq41*} to obtain
        {\small
	\begin{align}
	\int_{0}^{T}\bigg[(u_{t}(t),\chi(t))+(\Delta u_{t}(t),\Delta \chi(t))+\alpha(\nabla u(t), \nabla \chi(t))\bigg]dt
          &=\int_{0}^{T}\bigg[(\nabla \cdot g(u),  \chi(t))\bigg]dt . \label{eq43}
	\end{align}
        }
        A density argument ($M$ was arbitrary) asserts that \eqref{eq43} is true for every $\chi \in {L^{2}(I; H_0^2(\Omega))}$. Hence, the function $u$ satisfies \eqref{eq4} for $\chi \in H_0^2(\Omega),$ and $t\in[0,T]$.
	\item \textit{Initial data:} We seek to establish that $u(0)=u_0$. Integration by parts on \eqref{eq43} yields
	\begin{align}
          \int_{0}^{T}\bigg[(-u(t),\chi_t(t)) &+(\Delta u_{t}(t),\Delta \chi(t))+\alpha(\nabla u(t), \nabla \chi(t))\bigg]dt\nonumber\\
          &=\int_{0}^{T}\bigg[(\nabla \cdot g(u),  \chi(t))\bigg]dt+(u(0),\chi(0)), \label{eq44}
	\end{align}
	for every $\chi\in C^1([0,T]; H_0^2(\Omega))$ with $\chi(T)=0$. Applying a similar exercise to \eqref{eq42} yields
	\begin{align}
          \int_{0}^{T}\bigg[-(u^m(t),\chi_{t}(t)) &+ (\Delta u_{t}^m(t),\Delta \chi(t))+\alpha(\nabla u^m(t), \nabla \chi(t))\bigg]dt \nonumber\\
          &=\int_{0}^{T}\bigg[(\nabla \cdot g(u^m),  \chi(t))\bigg]dt+(u^m(0),\chi(0)).  \label{eq45}
	\end{align}
	Again using \eqref{eq41} and \eqref{eq41*} with $m=m_k$, taking the limit along the subsequence $\{u^{m_{k}}\}_{k=1}^{\infty}$, and using $u^m(0)\to u_0$ as $m \to\infty$ in the above leads to
	\begin{align}
	&&\int_{0}^{T}\bigg[-(u(t),\chi_{t}(t))+(\Delta u_{t}(t),\Delta \chi(t))+\alpha(\nabla u(t), \nabla \chi(t))\bigg]dt\nonumber\\&& =\int_{0}^{T}\bigg[(\nabla \cdot g(u),  \chi(t))\bigg]dt+(u_0,\chi(0)).  \label{eq46}
	\end{align}
        As $\chi(0)$ is arbitrary, comparing equation (\ref{eq44}) with (\ref{eq46}), we deduce that $u(0)=u_0$. We have thus established the existence of a solution $u$ to \eqref{eq:P1}, Problem $P^1$.
	\item \textit{Uniqueness:} Let us assume that $u_1$ and $u_2$ are two weak solutions to \eqref{eq4}. Setting $W:=u_1-u_2$, leads to 
	\begin{eqnarray}\left\{
	\begin{array}{ll}
	(W_{t}(t),\chi)+(\Delta W_{t}(t),\Delta \chi)=-\alpha(\nabla W(t), \nabla \chi) 
          -( g(u_1)-g(u_2),\nabla  \chi),\\
	W(0)={0}.
	\end{array}
	\right.\label{eq47}
	\end{eqnarray}
	Taking $\chi=W$ in (\ref{eq47}) and using Young's inequality, the Cauchy–Schwarz inequality, and the stability estimate for $g$ yields
	$$\frac{\mathrm{d}}{\mathrm{d}t}\big[\|W(t)\|^2+\|\Delta W(t)\|^2\big] \leq C\|W(t)\|_2^2.$$
	We integrate the above with respect to time $t$ to obtain
	$$\|W(t)\|^2+\|\Delta W(t)\|^2\leq\|W(0)\|^{2}_{2} + C\int_{0}^{t}\|W(s)\|^{2}_{2}  ds.$$
	Combining this with \cref{lemma:2.2} yields
		$$\|W(t)\|_2^2\leq C_1\|W(0)\|^{2}_{2} + C_2\int_{0}^{t}\|W(s)\|^{2}_{2}  ds.$$
        Finally, Gr{\"o}nwall’s lemma implies
	$$\|W(t)\|^{2}_{2} \leq C_1e^{C_2t}\|W(0)\|^{2}_{2}=0.$$
	Hence, $ u_{1} = u_{2}.$
\end{enumerate}
\end{proof}

\subsection{Semidiscrete Error Estimates }\label{ssec:P1-semi}
We now proceed to discretize space and prove error estimates for a semidiscrete scheme. 
Let $S_h$ be a finite-dimensional subspace of $H_0^2(\Omega)$ over a spatial mesh having small mesh parameter $h$ (i.e., $0<h<1$), with the following property (see \cite{ciarlet2002finite}) for some $r > 3$:
\begin{equation}
\inf_{\phi \in S_{h}}\|u-\phi\|_{2} \leq Kh^{r-2}\|u\|_{r},
\;\;u \in H_{0}^{2}(\Omega) \cap H^{r}(\Omega). \label{eqn13}
\end{equation}
Here, $r$ is the order of accuracy of $S_h$, coinciding with the demand on Sobolev smoothness for $u$. Generally, $S_h$ is comprised of piecewise polynomials having degree at most $r-1$ \cite[Page 4]{thomee2007galerkin}. In the sequel, if $J$ is a closed region in $\R^n$, we define
\begin{align*}
  P_r(J) \coloneqq \left\{ u: J \rightarrow \R \;\; \big|\;\; u \textrm{ is a polynomial of degree at most $r$} \right\}.
\end{align*}
A particular example of $S_h$ is the subspace $U_h$, defined as
{\small
\begin{align*}
  U_{h} = \left\{u_{h} \in C^{1}(\overline{\Omega})\; \big|\; u_{h}|_J \in P_{3}(J), \textrm{ for every } J \in {\cal T} _{h}, \textrm{ and }
  \;u_{h}|_{\partial \Omega}=0 ,\;\pfpx{u_h}{n}\big|_{\partial \Omega} = 0 
  \right\}, 
\end{align*}
}
where $\mathcal{T}_h$ is a triangulation  as defined in \cref{ssec:notation}.
In what follows we will work with the general subspace $S_h$, with the understanding that $U_h$ is the primary exemplar. Next, we define the bilinear map $A(\cdot, \cdot):H^2_0(\Omega)\times H^2_0(\Omega)\to \mathbb{R}$ as
\begin{equation}
A(u_{1},u_{2}):= (\Delta u_{1},\Delta u_{2}), \;\forall\; u_{1},u_{2} \in H_{0}^{2}(\Omega). \label{eqn48}
\end{equation}
The following properties for the bilinear map $A$ can be directly verified:
\begin{enumerate}[leftmargin=13pt]
	\item {\bf $H^2$-Boundedness:} The Cauchy–Schwarz inequality implies that there is a positive constant $C_1$ such that 
	\begin{equation}
	| A(u_{1},u_{2})| \leq C_1 \|u_{1}\|_{2} \|u_{2}\|_{2}, \; u_{1},u_{2} \in H_{0}^{2}(\Omega). \label{eqn16}
	\end{equation}
	\item {\bf $H^2$-Coercivity:} There exists a positive constant $C_2$ such that
	\begin{equation}
	A(u_{1},u_{1}) \geq C_2 \|u_{1}\|^{2}_{2}, \; u_{1} \in H_{0}^{2}(\Omega). \label{eqn17}
	\end{equation}
\end{enumerate}
Next, consider the $H^2$-orthogonal projection $\Pi_h$ onto the subspace $S_h$ satisfying the following property: 
\begin{equation}
A(u - \Pi_{h}u, \chi)=0 ,\;\quad \chi \in S_{h}, u \in H_{0}^{2}(\Omega) \cap H^{r}(\Omega).\label{eqn18}
\end{equation} 
Finally, we define the following continuous time semidiscrete Galerkin approximation to the problem \eqref{eq1*}:
Find $u_h : [0,T] \to S_h$ such that
\begin{eqnarray}\left\{
\begin{array}{ll}
(u_{t}^h(t),\chi)+(\Delta u_{t}^h(t),\Delta \chi)+\alpha(\nabla u^h(t), \nabla \chi)=-( g(u^h),\nabla  \chi),\\
(u^h(0),\chi)=(u_0^h,\chi),
\end{array}
\right.\label{eq77}
\end{eqnarray}
for all $\chi \in S_h$, where $u_0^h \in S_h$ is a suitable approximation to $u_0$, e.g., $u_0^h = \Pi_h u_0$.

\subsubsection{Preliminary lemmas}
We present two intermediate lemmas that are useful for proving our main result.
\begin{lemma}\label{lm1}
	Assume that $u(t), u_t(t)\in H_{0}^{2}(\Omega) \cap H^{r}(\Omega)$ for $t \in [0,T]$, then the following properties holds: 
	\begin{align}
	&\|u(t)- \Pi_{h}u(t)\|_{2} \leq   {K_1}h^{r-2}\|u(t)\|_{r}, \label{eqn25}\\
	&\|u_t(t)- \Pi_{h}u_t(t)\|_{2} \leq   {K_2}h^{r-2}\|u_t(t)\|_{r}.\label{eqn25*}
	\end{align}
	Where $r$ is the order of accuracy of $S_h$ defined above.
\end{lemma}
\begin{proof}
We apply the coercivity property \eqref{eqn17} and orthogonal projection \eqref{eqn18} to obtain
\begin{eqnarray}
C_1 \|u - \Pi_{h}u\|^{2}_{2}&\leq & A(u - \Pi_{h}u, u - \Pi_{h}u)\nonumber\\
& = & A(u - \Pi_{h}u, u ) = A(u - \Pi_{h}u, u - \phi), \; \phi \in S_h.\nonumber
\end{eqnarray}
The boundedness property \eqref{eqn16} implies 
\begin{equation}
\|u - \Pi_{h}u\|_{2} \leq K\inf_{\phi \in S_{h}}\| u - \phi\|_{2}. 
\end{equation} 
Finally, we use \eqref{eqn13} to obtain
\begin{eqnarray}
\|u - \Pi_{h}u\|_{2} \leq K_{1}h^{r-2}\|u\|_{r}. \label{eqn22}
\end{eqnarray}
Differentiating \eqref{eqn18} in time and performing a similar analysis results in \eqref{eqn25*}.
\end{proof}

\begin{lemma}[Stability estimate]\label{lm2}
	Let $u^h(t)$ be a solution of \eqref{eq77}. Then the following estimate holds:
        \begin{equation}\label{eq:P1-semi-stab-1}
	\|u^h(t)\|_{2}  \leq  C\|u_{0}\|_{2}  ,  \;\;t \in(0,T].
	\end{equation} 
	Moreover, there exists a constant $C_1 > 0$ such that
	\begin{equation}\label{eq:P1-semi-stab-2}
          \|u^h\|_{L^{\infty}([0,T];L^{\infty}(\Omega))}  \leq  C_1\|u_{0}\|_{2}.  
	\end{equation}
\end{lemma}
\begin{proof}
Setting $\chi = u^h$ in \eqref{eq77}, we have
\begin{equation}	
\dfrac{1}{2} \dfrac{\mathrm{d}}{\mathrm{d}t}\Big[\|u^h\|^{2} + \|\Delta u^h\|^{2}\Big]  +   \alpha \|\nabla u^h\|^{2}  = \int_{\Omega} (\nabla\cdot g(u^h))u^h dx .\label{eqn5**}
\end{equation}
Using \eqref{eqn7} in \eqref{eqn5**} along with the positivity of $\alpha$ yields,
$$ \dfrac{\mathrm{d}}{\mathrm{d}t}\Big[\|u^h\|^{2} + \|\Delta u^h\|^{2}\Big]\leq 0 . $$
Integrating the above with respect to time and using \cref{lemma:2.2} yields
\begin{equation}
\|u^h(t)\|_{2}  \leq  C\|u^h_{0}\|_{2}  ,  \;\;t \in(0,T],\nonumber
\end{equation} 
which establishes \eqref{eq:P1-semi-stab-1}. Finally, we use the fact that $H^{2}(\Omega)$ is embedded in $L^{\infty}(\Omega)$ (cf. Sobolev’s Inequality, \cite{quarteroni2009numerical}) in \eqref{eq:P1-semi-stab-1} to yield \eqref{eq:P1-semi-stab-2}.
\end{proof}

\subsubsection{Semidiscrete scheme convergence}
We can now present our main error estimate for the semidiscrete form of Problem $P^1$.
\begin{theorem}\label{thmerror}
  Let $r$ be the order of accuracy of the subspace $S_h$ with $r > 3$.  Assume that $u$ and $u_{h}$ are the solutions of equations (\ref{eq4}) and (\ref{eq77}), respectively and $u \in L^{\infty}(I;H_{0}^{2}(\Omega))~ \cap~ L^{2}(I;H^{r}(\Omega)) $ with $\partial_tu \in  L^{2}(I;H^{r}(\Omega)) $ and $u_{0} \in H_{0}^{r}(\Omega)$. 
  Then, there exists a positive constant $C$ such that
	\begin{equation}
          \|u-u^h\|_{L^{\infty}(I;H^{2}_{0}(\Omega))}^2+\|u-u^h\|_{L^{2}(I;H_{0}^{1}(\Omega))}^2\leq C(\|u_0-u_0^h\|_2^2+h^{2(r-2)}).\label{eq48}
	\end{equation}
\end{theorem}
\begin{proof}
Letting $\chi \in S_h$, we subtract \eqref{eq4} from \eqref{eq77}, to obtain
\begin{align}
&\big(u^h_t(t)-u_t(t),\chi\big)+\big(\Delta u_{t}^h(t)-\Delta u_{t}(t),\Delta \chi\big)+\alpha\big(\nabla u^h(t)-\nabla u(t), \nabla \chi\big)\nonumber\\&~~~~~~~~~~~~~~~~~~~~~~~~~~~~~~~~~~~~~~~~~~~~~~~~~~~~~~~~~~~~~ =-\big( g(u^h)-g(u), \nabla \chi\big). \label{eq49}
\end{align}
In what follows, we will suppress $t$-dependence and write $u^h = u^h(t)$ and $u = u(t)$ when no confusion arises.
Setting $\chi=u^h-\Pi_{h}u$ and using $u^h-\Pi_{h}u=(u^h-u)+(u-\Pi_{h}u)$ in \eqref{eq49} yields
\begin{align}
&\frac{1}{2}\frac{\mathrm{d}}{\mathrm{d}t}\bigg[\|u^h-u\|^2+\|\Delta( u^h-u)\|^2\bigg]+\alpha\|\nabla u^h-\nabla u\|^2\leq-\big(u^h_t-u_t,u-\Pi_{h}u\big)\nonumber \\&~~~~~~~~~~~~~-\big(\Delta u_{t}^h-\Delta u_{t},\Delta (u-\Pi_{h}u)\big)-\alpha\big(\nabla u^h-\nabla u, \nabla (u-\Pi_{h}u)\big)\nonumber\\&~~~~~~~~~~~~~-\big( g(u^h)-g(u), \nabla (u^h-u)\big)-\big( g(u^h)-g(u), \nabla (u-\Pi_{h}u)\big),  \label{eq50}
\end{align}
\begin{align}
\Rightarrow~~~&\frac{1}{2}\frac{\mathrm{d}}{\mathrm{d}t}\bigg[\|u^h-u\|^2+\|\Delta( u^h-u)\|^2\bigg]+\alpha\|\nabla u^h-\nabla u\|^2\leq-\partial_{t}\big(u^h-u,u-\Pi_{h}u\big)\nonumber\\&~~~~~~~~~~~+\big(u^h-u,\partial_t(u-\Pi_{h}u)\big)-\partial_{t}\big(\Delta u^h-\Delta u,\Delta (u-\Pi_{h}u)\big)\nonumber \\&~~~~~~~~~~~+\big(\Delta u^h-\Delta u,\Delta \partial_{t}(u-\Pi_{h}u)\big)-\alpha\big(\nabla u^h-\nabla u, \nabla (u-\Pi_{h}u)\big)\nonumber\\&~~~~~~~~~~~-\big( g(u^h)-g(u), \nabla (u^h-u)\big)-\big( g(u^h)-g(u), \nabla (u-\Pi_{h}u)\big).   \label{eq51}
\end{align}
Using the Cauchy–Schwarz inequality, Young's inequality, and the stability estimate for $g$ in the above equation, we conclude that there exist positive constants $C_1,C_2$, and $C_3$ such that
\begin{align}
&\frac{1}{2}\frac{\mathrm{d}}{\mathrm{d}t}\bigg[\|u^h-u\|^2+\|\Delta( u^h-u)\|^2\bigg]+\alpha\|\nabla u^h-\nabla u\|^2\leq -\partial_{t}\big(u^h-u,u-\Pi_{h}u\big)\nonumber\\& ~~~~~~~~~~~~~~~~~~~~~~-\partial_{t}\big(\Delta u^h-\Delta u,\Delta (u-\Pi_{h}u)\big)+C_1\|u^h-u\|_2^2\nonumber\\&~~~~~~~~~~~~~~~~~~~~~~+C_2\|u-\Pi_{h}u\|_2^2+C_3\|\partial_t(u-\Pi_{h}u)\|_2^2. \label{eq52}
\end{align}
Integrating the above inequality from $0$ to $t$ yields
\begin{align}
&\|u^h(t)-u(t)\|^2+\|\Delta (u^h(t)-u(t))\|^2+2\alpha\int_{0}^{t}\|\nabla u^h(s)-\nabla u(s)\|^2ds\leq\|u_0^h-u_0\|_2^2\nonumber\\&~~~ -2\big(u^h(t)-u(t),u(t)-\Pi_{h}u(t)\big)+2\big(u_0^h-u_0,u_0-\Pi_{h}u(0)\big)\nonumber\\&~~~-2\big(\Delta u^h(t)-\Delta u(t),\Delta (u(t)-\Pi_{h}u(t))\big)\nonumber\\&~~~+2\big(\Delta u_0^h-\Delta u_0,\Delta (u_0-\Pi_{h}u(0))\big)+C_4\int_{0}^{t}\|u^h(s)-u(s)\|_2^2ds\nonumber\\&~~~+C_5\int_{0}^{t}\|u(s)-\Pi_{h}u(s)\|_2^2ds+C_6\int_{0}^{t}\|\partial_t(u(s)-\Pi_{h}u(s))\|_2^2ds. \label{eq53}
\end{align}
Using \cref{lemma:2.2} in the above, we have
\begin{align}
&\|u^h(t)-u(t)\|_2^2+4\alpha\int_{0}^{t}\|\nabla u^h(s)-\nabla u(s)\|^2ds\leq 2\|u_0^h-u_0\|_2^2\nonumber\\&~~~ -4\big(u^h(t)-u(t),u(t)-\Pi_{h}u(t)\big)+4\big(u_0^h-u_0,u_0-\Pi_{h}u(0)\big)\nonumber\\&~~~-4\big(\Delta u^h(t)-\Delta u(t),\Delta (u(t)-\Pi_{h}u(t))\big)\nonumber\\&~~~+4\big(\Delta u_0^h-\Delta u_0,\Delta (u_0-\Pi_{h}u(0))\big)+C_7\int_{0}^{t}\|u^h(s)-u(s)\|_2^2ds\nonumber\\&~~~+C_8\int_{0}^{t}\|u(s)-\Pi_{h}u(s)\|_2^2ds+C_9\int_{0}^{t}\|\partial_t(u(s)-\Pi_{h}u(s))\|_2^2ds. \label{eq53*}
\end{align}
We now use the Cauchy–Schwarz inequality and Young's inequality to estimate several terms in \eqref{eq53*}:
  {\small
\begin{align*}
  -4\big(u^h-u,u-\Pi_{h}u\big) &\leq 4\|u^h-u\|\|u-\Pi_{h}u\|\leq\frac{1}{2}\|u^h-u\|^2+8\|u-\Pi_{h}u\|^2,\\
  4\big(u_0^h-u_0,u_0-\Pi_{h}u(0)\big) &\leq 4\|u_0^h-u_0\|\|u_0-\Pi_{h}u(0)\|
  \leq{2}\|u_0^h-u_0\|^2+{2}\|u_0-\Pi_{h}u(0)\|^2,\\ 
  -4\big(\Delta u^h-\Delta u,\Delta (u-\Pi_{h}u)\big) &\leq 4\|\Delta u^h-\Delta u\|\|\Delta(u-\Pi_{h}u)\|
  \leq\frac{1}{2}\|\Delta u^h-\Delta u\|^2+8\|\Delta(u-\Pi_{h}u)\|^2,\\
  4\big(\Delta u_0^h-\Delta u_0,\Delta (u_0-\Pi_{h}u(0))\big) &\leq 4\|
\Delta u_0^h-\Delta u_0\|\|\Delta(u_0-\Pi_{h}u(0))\|
  \leq{2}\|\Delta u_0^h-\Delta u_0\|^2 \\
  & \hskip 170pt +{2}\|\Delta(u_0-\Pi_{h}u(0))\|^2.
\end{align*}
  }
Using these in (\ref{eq53*}) and applying Gr{\"o}nwall's lemma yields
\begin{align}
&\|u^h(t)-u(t)\|_2^2+C_{10}\int_{0}^{t}\|\nabla u^h(s)-\nabla u(s)\|^2ds\leq C_{11}\bigg[\|u_0^h-u_0\|_2^2\nonumber\\&~~~~~~~~~~~~~~~~~~~~~+\|u_0-\Pi_{h}u(0)\|^2_2+\int_{0}^{T}\|u(t)-\Pi_{h}u(t)\|_2^2dt \nonumber\\&~~~~~~~~~~~~~~~~~~~~~+\|u(t)-\Pi_{h}u(t)\|_2^2+\int_{0}^{T}\|\partial_t(u(t)-\Pi_{h}u(t))\|_2^2dt \bigg]e^{C_8T}.\label{eq55}
\end{align}
  Taking the supremum over $t \in [0, T]$, we have
\begin{align}
&\sup_{0\leq t\leq T}\|u^h(t)-u(t)\|_2^2+C_{10}\int_{0}^{T}\|\nabla u^h(t)-\nabla u(t)\|^2dt\leq C_{11}\bigg[\|u_0^h-u_0\|_2^2\nonumber\\&~~~~~~~~~~~~~~~~~~~~~+\|u_0-\Pi_{h}u(0)\|^2_2+\int_{0}^{T}\|u(t)-\Pi_{h}u(t)\|_2^2dt \nonumber\\&~~~~~~~~~~~~~~~~~~~~~+\|u(t)-\Pi_{h}u(t)\|_2^2+\int_{0}^{T}\|\partial_t(u(t)-\Pi_{h}u(t))\|_2^2dt \bigg]e^{C_8T}.\label{eq56}
\end{align}
Using \cref{lm1}, we have
\begin{align}
&\|u(t)-\Pi_{h}u(t)\|_2^2~~\leq~~ C_9h^{2(r-2)}\|u(t)\|_r^2,\nonumber\\&\|u_0-\Pi_{h}u(0)\|_2^2~~\leq~~ C_9h^{2(r-2)}\|u_0\|_r^2,\nonumber\\&\int_{0}^{T}\|u(t)-\Pi_{h}u(t)\|_2^2dt~~\leq~~C_9h^{2(r-2)}\int_{0}^{T}\|u(t)\|_r^2dt,\nonumber\\&\int_{0}^{T}\|\partial_t(u(t)-\Pi_{h}u(t))\|_2^2dt~~\leq~~C_{10}h^{2(r-2)}\int_{0}^{T}\|\partial_tu(t)\|_r^2dt.\label{eq57}
\end{align}
Finally, using the theorem hypothesis along with estimates \eqref{eq57} in the expression \eqref{eq56} gives rise to 
\begin{align*}
&\sup_{0\leq t\leq T}\|u^h(t)-u(t)\|_2^2+C_{10}\int_{0}^{T}\|\nabla u^h(t)-\nabla u(t)\|^2dt\leq C\bigg(\|u_0^h-u_0\|_2^2\\
  &+h^{2(r-2)}\|u_0\|_r^2+h^{2(r-2)}\|u(t)\|_r^2 +h^{2(r-2)}\int_{0}^{T}\|u(t)\|_r^2dt+h^{2(r-2)}\int_{0}^{T}\|\partial_tu(t)\|_r^2dt\bigg),
\end{align*}
which shows,
\begin{equation*}
\sup_{0\leq t\leq T}\|u^h(t)-u(t)\|_2^2+C_{10}\int_{0}^{T}\|\nabla u^h(t)-\nabla u(t)\|^2dt\leq C(\|u_0^h-u_0\|_1^2+h^{2(r-2)}).
\end{equation*}
proving \eqref{eq48}.
\end{proof}

\subsection{Full Discretization of Problem $P^1$}\label{ssec:P1-fully}
In this section, we provide a fully discrete approach to approximate the weak solution of Problem $P^1$.
Let $N \in \N$, 
and define the temporal stepsize $k =\frac{T}{N}$.
Let $t^m=mk$ for $m = 0, 1, 2 \ldots N$. For a sequence $\{U^m\}_{m=0}^N \subset L^2(\Omega)$, we define 
\begin{eqnarray*}
	U^m=U(t^m)\;\;\; \mbox{and} \;\;\;\delta_t U^m=\frac{U^{m}-U^{m-1}}{k},\;\;\;  \forall \;\;m=1,\ldots,N.
\end{eqnarray*}
We will discretize our semidiscrete scheme \eqref{eq77} with backward Euler.
This leads to the following fully discrete scheme: for every $m \in [N]$, find $U^m$ that solves,
\begin{eqnarray}
  \left\{
\begin{array}{ll}
(\delta_t U^m,\chi)+(\Delta \delta_t U^m,\Delta \chi)+\alpha(\nabla U^{m}, \nabla \chi) =-( g( U^{m}),\nabla  \chi),                   &  \\
U^{0}=\Pi_hu(0).                                                                                                                               &
\end{array}
\right.\label{eq85}
\end{eqnarray}
As before, we first derive a stability estimate for this fully discrete problem. 
\begin{lemma} \label{lemma11}
  If $U^{m}$ is a solution of \eqref{eq85}, then there exists a constant $C$ (that may depend on $m$) such that, 
	\begin{equation*}
	\|U^{m}\|_{2} \leq C\|U^{0}\|_{2},\;\;  m \geq 1. 
	\end{equation*}
        Moreover, there exists a constant $C_1 > 0$ (possibly $m$-dependent) such that
	\begin{equation*}
	\|U^{m}\|_{L^{\infty}(\Omega)} \leq C_1\|U^{0}\|_{2},\;\;  m \geq 1. 
	\end{equation*}
\end{lemma}
\begin{proof}
Setting $\chi$ = $U^{m}$ in (\ref{eq85}), we have
  {\small
\begin{align}
&\bigg(\frac{U^{m}-U^{m-1}}{k},U^m\bigg)+\bigg(\Delta \frac{U^{m}-U^{m-1}}{k},\Delta U^m\bigg)+\alpha(\nabla U^{m}, \nabla U^m)
  =-( g( U^{m}),\nabla  U^m).\nonumber
\end{align} 		
  }
Now using the relations \eqref{eqn6} and \eqref{eqn7} yields
$$\|U^m\|^2+\|\Delta U^m\|^2\leq (U^{m-1},U^m)+(\Delta U^{m-1},\Delta U^m).$$
Using \cref{lemma:2.2} on the left-hand side above and the Cauchy-Schwarz and Young's inequalities on the right-hand side above yields
$$\|U^m\|_2\leq C\|U^{m-1}\|_2 .$$
As $m$ is arbitrary, hence we get the following inequality  
$$\|U^{m}\|_{2} \leq C\|U^{0}\|_{2},\;\;  m \geq 1.$$
We can prove the rest of the proof by using Sobolev Inequality \cite{quarteroni2009numerical}.
\end{proof}

Our next step is to the establish existence and uniqueness for the fully discrete scheme \eqref{eq85}. We require a particular fixed point theorem to accomplish this.

\begin{theorem}[Brouwer fixed point theorem, \cite{browder1965existence}]
  Let $H$ be a finite-dimensional Hilbert space having an inner product $(\cdot, \cdot)_H$ and induced norm $\|\cdot\|_H$. Also, assume that $ \Lambda:H\to H $ is a continuous operator satisfying $( \Lambda(\psi), \psi)_H>0$ for all $\psi\in H $ satisfying $\|\psi\|_H = a_{0}>0$. Then, there exists a $\psi^{\prime}\in H$ such that $\|\psi'\|_H < a_{0}$ and $\Lambda(\psi') = 0$.
\end{theorem}

We use this fixed point theorem to prove the well-posedness of the fully discrete scheme.
\begin{theorem}\label{thm:P1-full-wellposed}
	Suppose $U^{0},U^{1},\ldots,U^{m-1}$ is given  for each  $m>1$. Then, $U^{m}$ exists  and it satisfies (\ref{eq85}) uniquely.
\end{theorem}	
\begin{proof}
Using the given hypothesis, the initial term $U^{0}$ exists vacuously. Let us assume that $U^{0},U^{1},\ldots,U^{m-1}$ exist. Now using mathematical induction, we assert that $U^{m}$ also exists. To do the task, let us consider an operator  $\Lambda:S_h\to S_h $  defined by 
\begin{align}
&(\Lambda(\psi), \chi )= ( \psi ,\chi) +(\Delta \psi,\Delta\chi) - ( U^{m-1} ,\chi) - (\Delta U^{m-1},\Delta\chi) +{{k}}\bigg(( g(\psi) ,\nabla\chi)\nonumber\\&~~~~~~~~~~~~~~~~~~~~~~~~~~~~~~~~~~~~~~~~~~~~~~~~~~~~~~~~+\alpha(\nabla \psi,\nabla\chi)\bigg),\;\;\;\; \forall \;\psi,\chi \in S_{h}.\nonumber
\end{align}
One can easily see that $\Lambda$ is a continuous map. 

\noindent Putting $\chi = \psi$ with  relation (\ref{eqn7}) in the above equation to obtain
\begin{eqnarray}
(\Lambda(\psi), \psi)= \|\psi\|^{2}  + \|\Delta \psi\|^{2} - ( U^{m-1} ,\psi) - (\Delta U^{m-1},\Delta \psi)+ k  \alpha\|\nabla \psi\|^{2} ,\nonumber
\end{eqnarray}
which implies
\begin{eqnarray}
(\Lambda(\psi), \psi)&\geq&\|\psi\|^{2}  + \|\Delta \psi\|^{2} - ( U^{m-1} ,\psi) - (\Delta U^{m-1},\Delta \psi)\nonumber\\
&\geq& C_1(\|\psi\|^{2}_{2} - \|U^{m-1}\|^2- \|\Delta U^{m-1}\|^2).\nonumber
\end{eqnarray}	
For $\|\psi\|_{2}^2 = \|U^{m-1}\|^2 +\|\Delta U^{m-1}\|^2+ C_2$ (choosing a suitable positive constant $C_2$ that satisfies the condition of the fixed point theorem), one can easily see that $(\Lambda(\psi), \psi) > 0 $. Consequently, Brouwer fixed point theorem guarantees the existence of $\psi^{\prime} \in S_{h}$ with $\Lambda(\psi^{\prime}) = 0$ such that  $\|\psi^{\prime}\|_2 \leq \|U^{m-1}\|^2 +\|\Delta U^{m-1}\|^2+ C_2$. Now, selecting  $U^{m} = \psi^{\prime}$ with $\Lambda(\psi^{\prime}) = 0$ satisfies equation (\ref{eq85}) and thus $U^{m}$ exists.

To prove the uniqueness part, we assume that $U^{m}_{1}$ and $U^{m}_{2}$ are the two distinct solutions of (\ref{eq85}). Choosing $W^{m}:= U^{m}_{1} - U^{m}_{2}$, we get
\begin{align}
(\delta_{t}W^{m},\chi) +(\Delta \delta_{t}W^{m},\Delta\chi)+\alpha(\nabla W^{m},\nabla \chi) = -(g(U_{1}^{m})-g(U_{2}^{m}),\nabla \chi). \label{eqn27*}
\end{align}

Now, we again use the induction method. Let us first assume that $W^{m-1} = 0$ and we will prove that $W^{m} = 0.$ For that, setting $\chi =  W^{m}$ in (\ref{eqn27*}) yields
\begin{eqnarray}
\frac{1}{2} \delta_{t} \Big(\|W^m\|^{2} + \|\Delta W^m\|^{2}\Big) 
\leq\|g(U_{1}^{m})-g(U_{2}^{m})\| \|\nabla W^{m}\|.  \label{eqn28*}
\end{eqnarray}	
Further, applying stability estimate for function $g$, 
we obtain
\begin{equation}
\|g(U_{1}^{m})-g(U_{2}^{m})\| \leq C_{1} \| W^{m}\|.\label{eqn29*}
\end{equation}
From equations (\ref{eqn28*}) and (\ref{eqn29*}), and \cref{lemma:2.2}, we get 
$$\delta_{t} \|W^m\|_{2}^{2}\leq C_{2}\big( \|W^{m}\|_{2}^{2}\big).$$
Further, utilizing the definition $\delta_{t} \|W^m\|_{1}^{2}$, we get$$ \|W^m\|_{2}^{2}\leq \frac{1}{1-C_2k}\big( \|W^{m-1}\|_{2}^{2}\big).$$
Now, selecting sufficiently small k such that $1-C_2k>0$ and using the  condition $W^{m-1} = 0$, we get $\|W^m\|_{2} = 0.$ Consequently, $W^m$ = 0 and hence we obtained the uniqueness.
\end{proof}

Finally, we establish an error estimate for this fully discrete scheme. 
We will state convergence in terms of norms on time-discrete versions of B{\^o}chner space norms. In particular, if $\{f_m\}_{m=0}^N \subset \mathcal{Z}$ for some Banach space $\mathcal{Z}$, then the following is a discrete $L^p$ norm:
\begin{align*}
  \|f\|^p_{L^p_{dis}(I; \mathcal{Z})} \coloneqq \sum_{m=0}^N  \|f^m\|^p_{\mathcal{Z}},
\end{align*}
for $1 \leq p < \infty$, and for $p = \infty$ the norm is the maximum over $m$ of $\|f^m\|_{\mathcal{Z}}$. We can now state our convergence theorem.

\begin{theorem}\label{thmBE}
	Assume that $u$ and $U^m$ are the solutions of equation (\ref{eq4}) and (\ref{eq85}), respectively and $u \in L^{\infty}(I;H_{0}^{2}(\Omega)) \cap L^{2}(I;H^{r}(\Omega)) $ with $u_{0} \in H_{0}^{r}(\Omega)$, $u_t \in  L^{1}(I;H^{r}(\Omega)) $ and $u_{tt} \in  L^{1}(I;H^{2}(\Omega)) $, where $r$ is the order of accuracy of $S_h$ with $r > 3$. Then, there exists a positive constant $C$ such that the following estimate in discrete $L^{2}(I;H_{0}^{2}(\Omega))$ norm holds:
	\begin{equation}
	\|u-U\|_{L^{2}_{dis}(I;H_{0}^{2}(\Omega))}\leq C(u,T)(h^{r-2}+k),\label{eq4*}
	\end{equation}
	where $$C(u,T)= \bigg[{K_1}\bigg(\sum_{m=1}^{N}\|u(t^m)\|_r+\int_{0}^{T}\|{u_t}(\cdot,s)\|_rds+\int_{0}^{T}\|{u_{tt}}(\cdot,s)\|_2ds\bigg)\bigg].$$
	Moreover, there exists another positive constant $K_2$ such that 
	\begin{align}
	\|u-U\|_{L^{\infty}_{dis}(I;H_{0}^{2}(\Omega))}\leq K_2(h^{r-2}+k). \label{Be*}
	\end{align}
\end{theorem}
\begin{proof}
We write the error in the following form $$u(t^m)-U^m= (u(t^m)-\Pi_hu(t^m))+(\Pi_hu(t^m)-U^m)=\rho^m+\theta^m.$$
From lemma \ref{lm1}, we already know $\rho^m$. So, we only need to find the bound of $\theta^m$.
To do the task, we subtract equation \eqref{eq85} from \eqref{eq4} to get 
\begin{align}
&\int_{\Omega}\bigg(\dot{u}(t^{m})-\big(\frac{U^m-U^{m-1}}{k}\big)\bigg)\chi d\Omega + \int_{\Omega}\Delta\bigg(\dot{u}(t^{m})-\big(\frac{U^m-U^{m-1}}{k}\big)\bigg)\Delta\chi d\Omega\nonumber\\&~~~~~~~~~~~~~~~~+\alpha\int_{\Omega}\nabla\big(u(t^m)-U^m\big)\nabla\chi d\Omega=\int_{\Omega}\big(g( U^{m})-g(u({t^m}))\big)\nabla\chi d\Omega.\nonumber
\end{align}
Solving the above, we have 
\begin{align}
&\int_{\Omega}\bigg(\frac{\theta^m-\theta^{m-1}}{k}\bigg)\chi d\Omega + \int_{\Omega}\Delta\bigg(\frac{\theta^m-\theta^{m-1}}{k}\bigg)\Delta\chi d\Omega + \alpha \int_{\Omega}\nabla \theta^m\nabla \chi d \Omega \nonumber \\&= -\int_{\Omega}(\sigma^m)\chi d\Omega 
- \int_{\Omega}\Delta(\sigma^m)\Delta\chi d\Omega -\alpha\int_{\Omega}\nabla(\rho^m )\nabla \chi d\Omega\nonumber\\&~~~~~~~~~~~~~~~~~~~~~~~~~~~~~~~~~~~~~~~~~~~ + \int_{\Omega}\big(g( U^{m})-g(u({t^m}))\big)\nabla\chi d\Omega. \label{eq86}
\end{align}
Where, $\sigma^m = \dot{u}(t^{m})-\frac{\Pi_hu(t^m)-\Pi_hu(t^{m-1})}{k}$.

\noindent Now, using stability estimate for function $g$, we have
$$\|g( U^{m})-g(u({t^m}))\|\leq C_1\|U^{m}-u({t^m})\|\leq C_1\big(\|\rho^m\|+\|\theta^m\|\big).$$
Using \cref{lemma:2.2}, Cauchy–Schwarz inequality and the above relation in equation \eqref{eq86} with $\chi=\theta^m$, we get
\begin{align}
&(1-C_1k)\|\theta^m\|_2\leq C_2(\|\theta^{m-1}\|_2+k\|\sigma^m\|_2+k\|\rho^m\|_2 ).\label{eq87}
\end{align}
Further, we add and subtract $\frac{u(t^m)-u(t^{m-1})}{k}$ in $\sigma^m$ to obtain
\begin{align}
\sigma^m = &\bigg(\dot{u}(t^{m})-\frac{u(t^m)-u(t^{m-1})}{k}\bigg)\nonumber\\&+\bigg(\frac{u(t^m)-\Pi_hu(t^m)}{k}-\frac{u(t^{m-1})-\Pi_hu(t^{m-1})}{k}\bigg)\nonumber\\& =\sigma^m_1+\sigma^m_2. \label{eq88}
\end{align}
Selecting $k_1$ with $0\leq k\leq k_1$ such that $1-C_1k\geq0$ and using the above relation while taking the sum from $1$ to $m$ in equation \eqref{eq87}, we get  
\begin{align}
&\|\theta^m\|_2\leq C_3\bigg(\|\theta^{0}\|_2+k\sum_{i=1}^{m}\|\sigma^i_1\|_2+k\sum_{i=1}^{m}\|\sigma^i_2\|_2+k\sum_{i=1}^{m}\|\rho^i\|_2 \bigg).\label{eq89}
\end{align}
$U^{0}=\Pi_hu(0)$ implies $\theta^{0}=0$. Now, using Taylor’s theorem in the integral form \cite{thomee2007galerkin} up to three terms, it follows that
$$\dot{u}(t^{m})-\frac{u(t^m)-u(t^{m-1})}{k}=-\frac{1}{k}\int_{t^{m-1}}^{t^{m}}(t^{m-1}-s)\ddot{u}(\cdot,s)ds.$$
This gives rise to
\begin{align}
&k\sum_{i=1}^{m}\|\sigma^i_1\|_2\leq\sum_{i=1}^{m}\| \int_{t^{m-1}}^{t^{m}}(t^{m-1}-s)\ddot{u}(\cdot,s)ds\|_2\leq k\int_{0}^{t^m}\|\ddot{u}(\cdot,s)\|_2ds.\label{eq90}
\end{align}
Also, we can see that 
$$\frac{u(t^m)-\Pi_hu(t^m)}{k}-\frac{u(t^{m-1})-\Pi_hu(t^{m-1})}{k}=\frac{1}{k}\int_{t^{m-1}}^{t^m}(I-\Pi_h)
\dot{u}(\cdot,s)ds.$$ 
Using Lemma  \ref{lm1} in above, we have
\begin{align}
&k\sum_{i=1}^{m}\|\sigma^i_2\|_2\leq\sum_{i=1}^{m} \int_{t^{m-1}}^{t^{m}}C_4h^{r-2}\|\dot{u}(\cdot,s)\|_rds\leq C_4h^{r-2}\int_{0}^{t^m}\|\dot{u}(\cdot,s)\|_rds.\label{eq91}
\end{align}
Again, using Lemma  \ref{lm1}, equations \eqref{eq90} and \eqref{eq91} in \eqref{eq89} yields
\begin{align}
\|\theta^m\|_2\leq \tilde{C}(u,T)(h^{r-2}+k),\label{eq92}
\end{align}
where $$\tilde{C}(u,T)= \bigg[C_5\bigg(\sum_{m=1}^{N}\|u(t^m)\|_r+\int_{0}^{T}\|\dot{u}(\cdot,s)\|_rds+\int_{0}^{T}\|\ddot{u}(\cdot,s)\|_2ds\bigg)\bigg].$$

\noindent Now, using the estimate of Lemma \ref{lm1} and equation \eqref{eq92} with triangle inequality yields
\begin{align}
\|u(t^m)-U^m\|_2\leq C(u,T)(h^{r-2}+k),\;\; m\in [N]. \label{Be}
\end{align}
Taking the sum over each time step $t^m$, we get an error in discrete $L^{2}(I;H_{0}^{2}(\Omega))$ norm as follows
\begin{align}
\|u-U\|_{L_{dis}^{2}(I;H_{0}^{2}(\Omega))}\leq C(u,T)(h^{r-2}+k).
\end{align}
Here, we have used the fact that $(a^2+b^2)^{1/2}\leq a+b$ when $a,b>0$. Again, we can take the supremum in equation \eqref{Be} to prove the second part of the theorem.
\end{proof}

\section{Analysis of Problem $P^2$ } \label{Sec P2}
In this section we present similar theoretical and algorithm results as for Problem $P^1$ above. Since the approach is materially identical, we leave details in the Supplementary Material, and here only state the major results. In particular: We give an existence and uniqueness result for \eqref{eq:P2} in \cref{exist1}, and provide the proof in \cref{existmixed}. We follow this with a convergence estimate for a semi-discrete scheme in \eqref{thmerror.} whose proof is contained in \cref{semimix}.

Finally, we discretize the time variable and provide well-posedness statements and convergence rates for a fully discrete scheme in \cref{fully-mix}.

Our first result is a well-posedness result for the mixed formulation \eqref{eq:P2} that parallels \cref{thm:P1-wellposed} for the primal formulation.
\begin{theorem}[Problem $P^2$ existence and uniqueness]\label{exist1} Assume that $(u_{0},p_{0}) \in H_0^1(\Omega)\times H_0^1(\Omega) $. Then, for $T>0$, there exists a unique weak solution $(u,p) $ to the problem (\ref{eq1*}) in the sense of the weak formulation \eqref{eq5}.
\end{theorem}
See \cref{existmixed} for the proof. We next obtain error estimates for a spatially discrete scheme over a bounded domain $\Omega\subseteq\mathbb{R}^{n}$ with n $\geq$ 2. Let $S_h$ be a finite-dimensional subspace of $H_0^1(\Omega)$ with small mesh parameter $h$ $ (0<h<1)$, with the following property (see \cite{thomee2007galerkin}) for $r\geq2$: There exists a constant C for $u\in H^r(\Omega) \cap H^1_0(\Omega)$ such that
\begin{equation}
\inf_{\chi \in S_{h}}\{\|u-\chi\|+h\|\nabla(u-\chi)\|  \}\leq Ch^{r}\|u\|_{r}.
\end{equation}
For $r=2$, consider the space of continuous piecewise polynomials ($P_1)$ as an example of such a type of subspace. In the case $r>2$, $S_h$ generally comprises of
piecewise polynomials having degree at most $r-1$ denoted as $P_{r-1}$ (Page 4, \cite{thomee2007galerkin}). Now, we define the following continuous time semidiscrete Galerkin approximation to the problem \eqref{eq1*}-\eqref{eq3}.

Find $u_h, p_h : [0,T] \to S_h$ such that
        \begin{subequations}\label{eq:P2-semi}
\begin{align}
	(u^h_{t}(t),\chi)+(\nabla p^h_{t}(t),\nabla \chi)+\alpha(p^h(t),  \chi)=(\nabla\cdot g(u^h),  \chi), &\quad \;\chi \;\in S_h,\label{eq77.a}\\
	(\nabla u^h(t),\nabla \chi')=(p^h(t),  \chi'), &\quad \;\chi' \;\in S_h,\label{eq77.b}\\
u^h(0)=u_0^h,
\end{align}\label{eq77.}
	\end{subequations}
where $u_0^h \in S_h$ is a suitable approximation to $u_0$. Our semi-discrete error estimate for the mixed formulation that parallels \cref{thmerror} is as follows:
\begin{theorem}\label{thmerror.}
	Assume that $(u,p)$ and $(u^h,p^h)$ are the solutions of equations (\ref{eq5}) and (\ref{eq77.}), respectively and $u \in L^{\infty}(I;H_{0}^{1}(\Omega))~ \cap~ L^{2}(I;H^{r}(\Omega)) $ with $u_t \in  L^{2}(I;H^{r}(\Omega)) $,
 $p \in  L^{\infty}(I;H^{1}_0(\Omega))\cap L^{\infty}(I;H^{r}(\Omega))\cap L^{2}(I;H^{r}(\Omega)) $ with $p_t \in  L^{2}(I;H^{r}(\Omega)) $ and $p(0) \in H^{r}(\Omega)$, where $r$ is the order of accuracy of $S_h$ with $r\geq 2$ defined above. Also, suppose that $u^h(0)=\Pi_hu(0)$. Then, there exists a positive constant $C$ with the following condition
	\vspace{0.2cm}
	\begin{equation}
	\|u-u^h\|_{L^{\infty}(I;L^{2}(\Omega))}+ \|p-p^h\|_{L^{\infty}(I;L^{2}(\Omega))}+h\|u-u^h\|_{L^{\infty}(I;H_{0}^{1}(\Omega))}\leq Ch^r.\label{eq48.}
	\end{equation}
\end{theorem}

We next provide a fully discrete approach to approximate the weak solution (\ref{eq5}) by employing the same notations that are used in section \ref{ssec:P1-fully}.
We employ the backward Euler approach for the fully discretization of equation \eqref{eq77.}: Find the approximate solution  $(U^{m}, P^m)$ of $(u(t^{m}), p(t^{m}))$ such that
\begin{subequations}\label{eq:P2-full}
\begin{align}
(\delta_t U^m,\chi)+(\nabla \delta_t P^m,\nabla \chi)+\alpha( P^{m},  \chi) &=( \nabla\cdot g( U^{m}),  \chi),                \label{eq85.a}     \\
 (\nabla U^m,\nabla \chi')&=( P^m,  \chi') ,                                    \label{eq85.b}         \\
U^{0}&=u_0^h.                                                                                                                        \end{align} \label{eq85.}
\end{subequations}

The following result provides existence and uniqueness of solution $U^m$ for the fully discrete scheme (\ref{eq85.}), and is analogous to \cref{thm:P1-full-wellposed}.
\begin{theorem}\label{thm:P2-full-wellposed}
	Suppose $(U^{0},P^0),(U^{1},P^1),\ldots,(U^{m-1},P^{m-1})$ is given  for each  $m>1$. Then, $(U^{m},P^m)$ exists  and satisfies (\ref{eq85.}) uniquely.
\end{theorem}	
The proof is contained in \cref{fully-mix}.
Our last major result is an error estimate that parallels \cref{thmBE}:
\begin{theorem}\label{thmBE*}
	Assume that $u$ and $U^m$ are the solutions of equation (\ref{eq5}) and (\ref{eq85.}), respectively and $u \in L^{\infty}(I;H_{0}^{1}(\Omega)) \cap L^{\infty}(I;H^{r}(\Omega)) $ with  $u_t \in  L^{2}(I;H^{r}(\Omega)) $ and $u_{tt} \in  L^{2}(I;L^{2}(\Omega)) $. Further, assume that $p \in L^{\infty}(I;H_{0}^{1}(\Omega)) \cap L^{\infty}(I;H^{r}(\Omega)) $ with  $p_t \in  L^{2}(I;H^{r+1}(\Omega)) $ and $p_{tt} \in  L^{2}(I;H^{1}(\Omega)) $. Also, $U^0=\Pi_hu(0)$ and $p(0)\in H^r(\Omega)$. Then, there exists a positive constant $C$ such that the following estimate holds:
	\begin{equation}
	\|u(t^J)-U^J\|+\|p(t^J)-P^J\|+h\|u(t^J)-U^J\|_1\leq C(h^{r}+k), \; J=1,2\ldots N.\nonumber
	\end{equation}
\end{theorem}
The proof is contained in \cref{fully-mix}.

\section{Numerical Experiments} \label{Num}
We validate the theoretical results using  several examples. All numerical solutions are computed using the mixed formulation corresponding to formulations \eqref{eq:P2} and \eqref{eq:P2-full}. We employ the following norms for the error $e(t) = u(t)- u_h(t)$  at final time $t=t^N$, where $u$ and $u_h$ are the exact and numerical solution, respectively..
\begin{align}
&\|e(t)\|_{L^{2}} := \Big(\int_{\Omega}|u(x,t)- u_h(x,t)|^{2}dx\Big)^{\frac{1}{2}},\nonumber\\
&\|e(t)\|_{H^1}: = \Big(\sum_{k=0}^{1}  |e(t)|^{2}_{H^k}\Big)^{\frac{1}{2}} = \Big( \|e(t)\|_{L^{2}}^{2}+\|\nabla e(t)\|_{L^{2}}^{2}\Big)^{\frac{1}{2}},\nonumber\\
&\|e(t)\|_{H^2}: = \Big(\sum_{k=0}^{2}  |e(t)|^{2}_{H^k}\Big)^{\frac{1}{2}} = \Big( \|e(t)\|_{L^{2}}^{2}+\|\nabla e(t)\|_{L^{2}}^{2}+\|\Delta e(t)\|_{L^{2}}^{2}\Big)^{\frac{1}{2}},\nonumber\\
&\|e(t)\|_{L^{\infty}} := \max_{i} |u(x_{i},t)- u_h(x_{i},t)|.\nonumber
\end{align}
\begin{figure}[htpb]
	\centering
	\includegraphics[width=3.5cm, height=3.5cm]{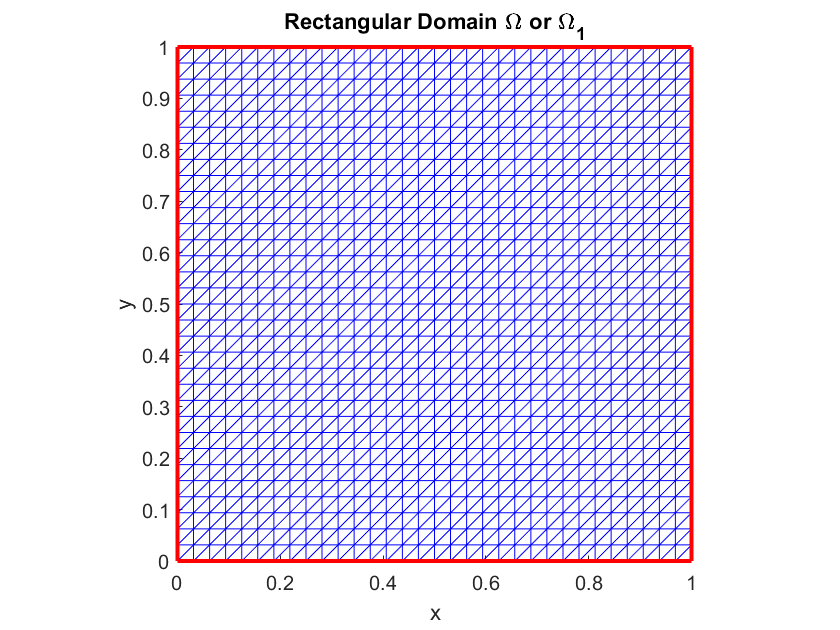}\;
	\includegraphics[width=3.5cm, height=3.5cm]{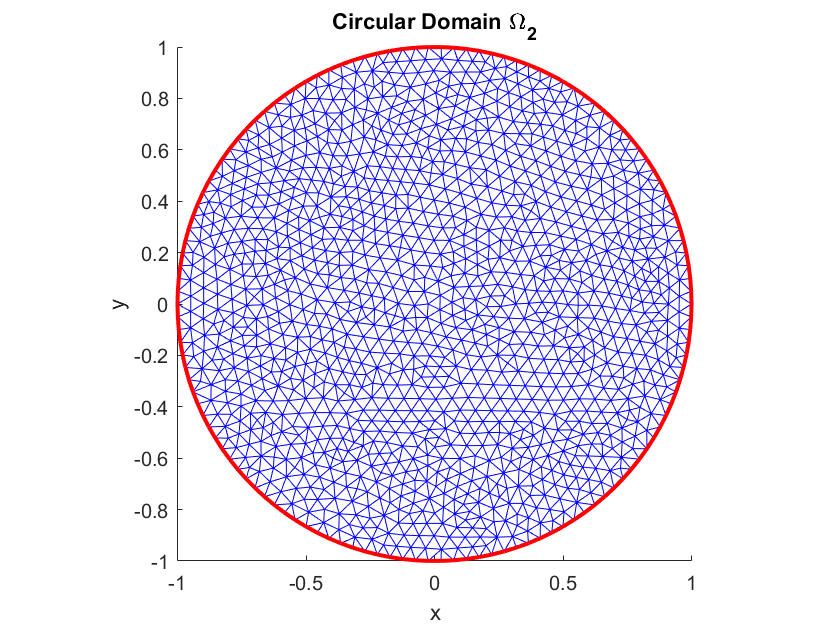}\;
		\includegraphics[width=3.5cm, height=3.5cm]{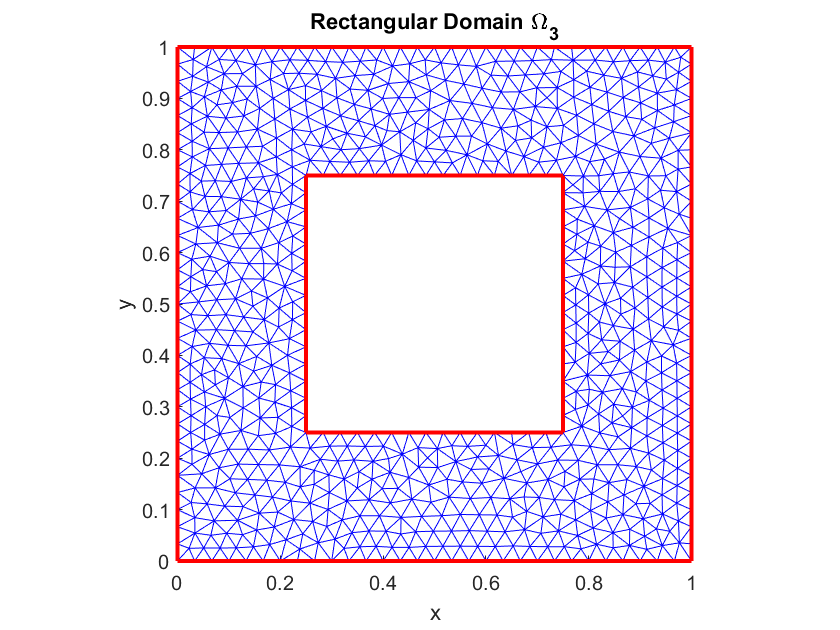}\\
			\includegraphics[width=3.5cm, height=3.5cm]{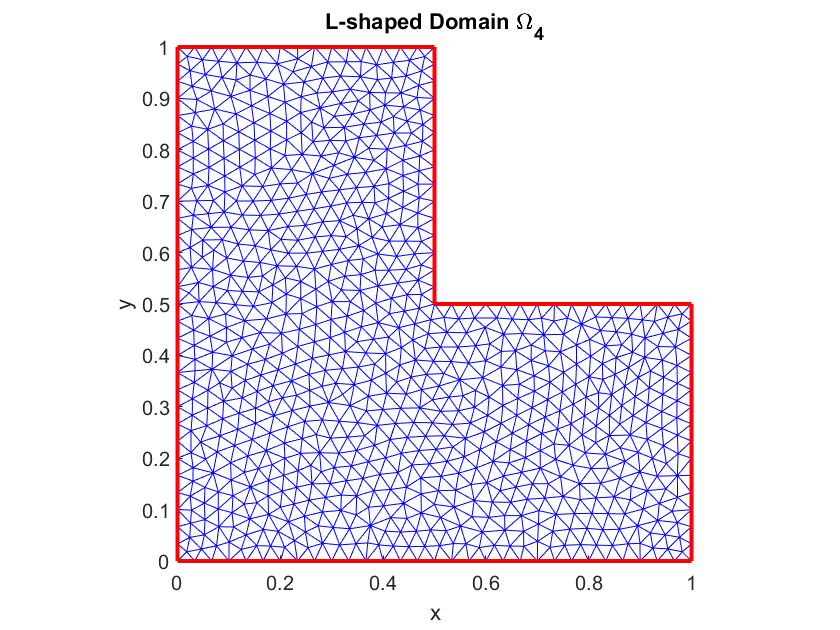}\;
				\includegraphics[width=3.5cm, height=3.5cm]{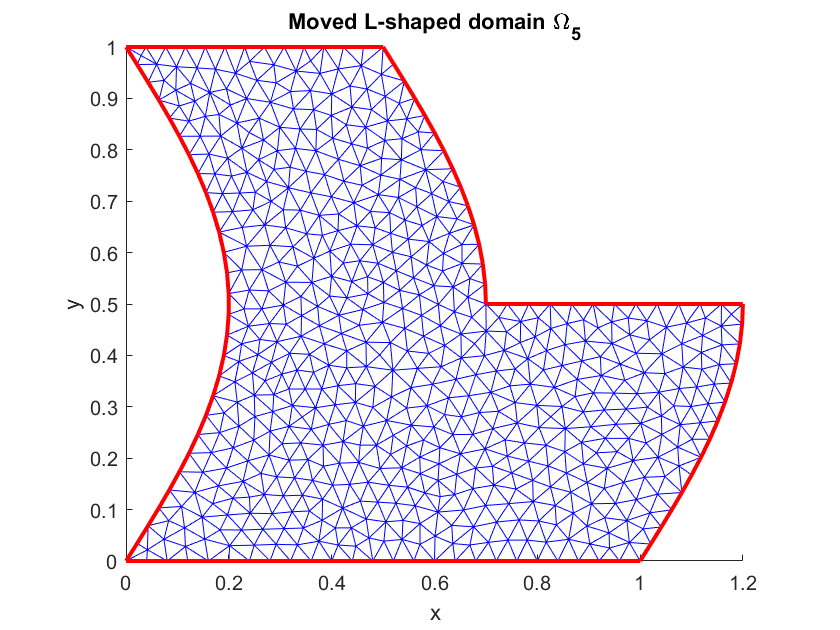}\;
					\includegraphics[width=3.5cm, height=3.5cm]{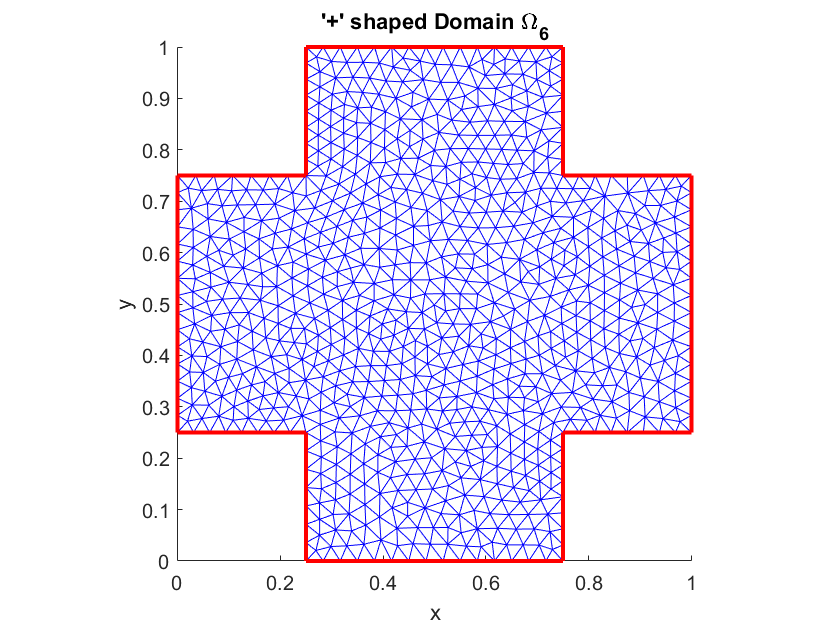}\\
        \caption{Six different domains $\Omega_i$, $i \in [6]$, considered in the examples.} \label{figg0}
\end{figure}
Computationally observed orders of convergence are computed in the standard way,
$$C^{1}-\mbox{order}= \frac{\log\big(\frac{\|e_1\|}{\|e_2\|}\big)}{\log\big(\frac{h_1}{h_2}\big)}, \;\;\;\;C^{2}-\mbox{order}= \frac{\log\big(\frac{\|e_1\|}{\|e_2\|}\big)}{\log\big(\frac{k_1}{k_2}\big)}.$$
Here, $e_1$ and $e_2$ are the errors corresponding to different step sizes $h_1, h_2$ in
spatial variable and $k_1, k_2$ in temporal variable. Several domains are considered in this paper. The rectangular domain $\Omega_1$ in Figure \eqref{figg0} is referred to as $\Omega$ unless and otherwise specified.
\begin{exm}\label{exx1}{Propagation of solitary wave in one-dimension}
  \rm
		We consider the one-dimensional Rosenau-Burgers model \eqref{eq1*},
		 $$u_{t}+u_{xxxxt} - u_{xx}+u_x+ uu_x =f\; \ \mbox{in}\;  [0,1]\times(0,T],\;T>0,$$
		where $f$ and the initial data are chosen so that the exact solution is given by,
		$$u(x,t)=e^{-t}x^3(1-x)^3,$$
		and boundary conditions $u(0,t)= u(1,t)=u_{xx}(0,t)=u_{xx}(1,t)=0$. 
                \Cref{tb1} (left) reveals the $L^2$ and $H^1$ errors and convergence rate using the BE scheme and Taylor-Hood finite element ${P_2}\times P_1$ with $k=0.01$ at time $T=1$. Table \eqref{tb1} (right) demonstrates the $C^{1}-\mbox{order}$ and $C^{2}-\mbox{order}$ with different $h$ and $k$ at time $T=1$ for ${P_2}\times P_1$ elements.  These results support the theory developed in Theorems \ref{thmerror.} and \ref{thmBE*}. 
		
		\begin{table}[htbp]
                  \begin{minipage}[b]{0.55\linewidth}
			\renewcommand{\arraystretch}{1.3}
			\centering
                        \resizebox{\textwidth}{!}{
			\begin{tabular}{c c c a a}
				$ {h} $ &   $\|e(t)\|_{L^{2}}$&$C^{1}-\mbox{order}$ &  $\|e(t)\|_{H^{1}}$&$C^{1}-\mbox{order}$ \\\toprule
				$1/4$ & ${4.4381 \times 10^{-6}}$ & $-$& ${1.2322 \times 10^{-4}}$ & ${ -}$ \\
				$1/8$ & ${7.8418 \times 10^{-7}}$ & ${2.5007}$& ${4.3693 \times 10^{-5}}$ & ${1.4958}$ \\
				$1/16$ & ${1.0607 \times 10^{-7}}$ & ${2.8861}$& ${1.1833 \times 10^{-5}}$ & ${ 1.8846}$ \\
				$1/32$ & ${1.3516 \times 10^{-8}}$ & ${2.9723}$& ${3.0165 \times 10^{-6}}$ & ${1.9719}$ \\
				$1/64$ & ${1.6977 \times 10^{-9}}$ & ${2.9931}$& ${7.5779 \times 10^{-7}}$ & ${1.9930}$ 
			\end{tabular}
                      }
                  \end{minipage}\hfill
                  \begin{minipage}[b]{0.42\linewidth}
			\renewcommand{\arraystretch}{1.3}
			\centering
                        \resizebox{\textwidth}{!}{
			\begin{tabular}{c c c }
				$h=k$ &   $\|e(t)\|_{H^{2}}$&$C^{1}-\mbox{order}, C^{2}-\mbox{order}$\\\toprule
				$1/4$ & ${4.2996 \times 10^{-3}}$ & $-$  \\
				$1/8$ & ${2.7737 \times 10^{-3}}$ & $0.6324$\\
				$1/16$ & ${1.4750 \times 10^{-3}}$ & ${0.9111}$  \\
				$1/32$ & ${7.4882 \times 10^{-4}}$ & ${0.9780}$  \\
				$1/64$ & ${3.7582 \times 10^{-4}}$ & ${0.9946}$ 
			\end{tabular}
                      }
                  \end{minipage}
                  \caption{\Cref{exx1} orders of convergence with $P_2 \times P_1$ finite elements. Left: $L^2$ and $H^1$ errors and computational order with $k=0.01$, $T=1$ . Right: $C^{1}-\mbox{order}$ and $C^{2}-\mbox{order}$ with $T=1$.}\label{tb1}
		\end{table}

%
%

			\begin{figure}[htbp]
			\centering
			\includegraphics[width=0.24\textwidth]{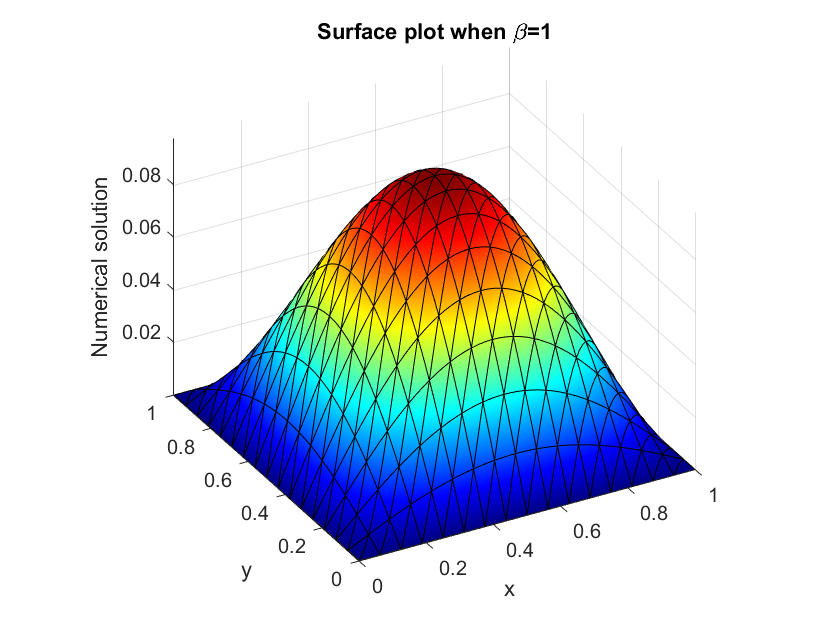}
			\includegraphics[width=0.24\textwidth]{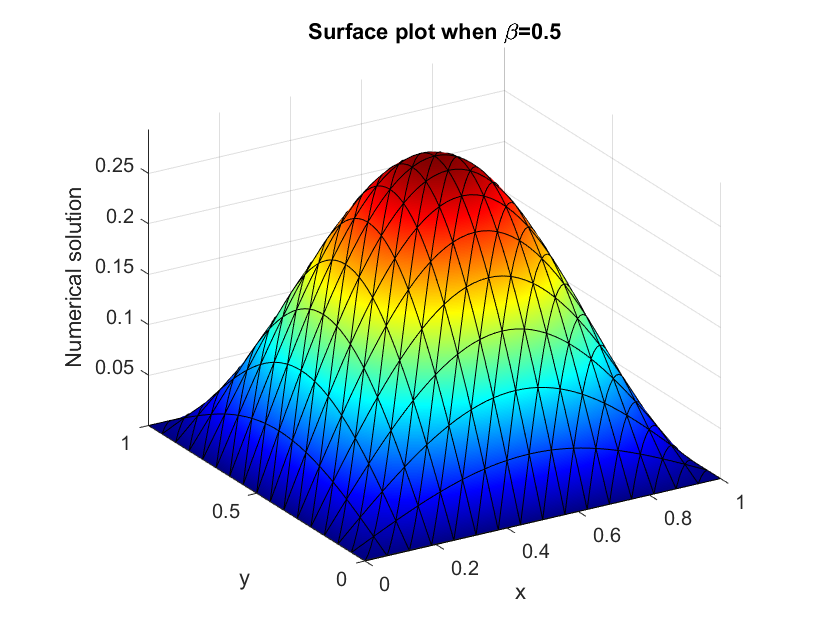}
			\includegraphics[width=0.24\textwidth]{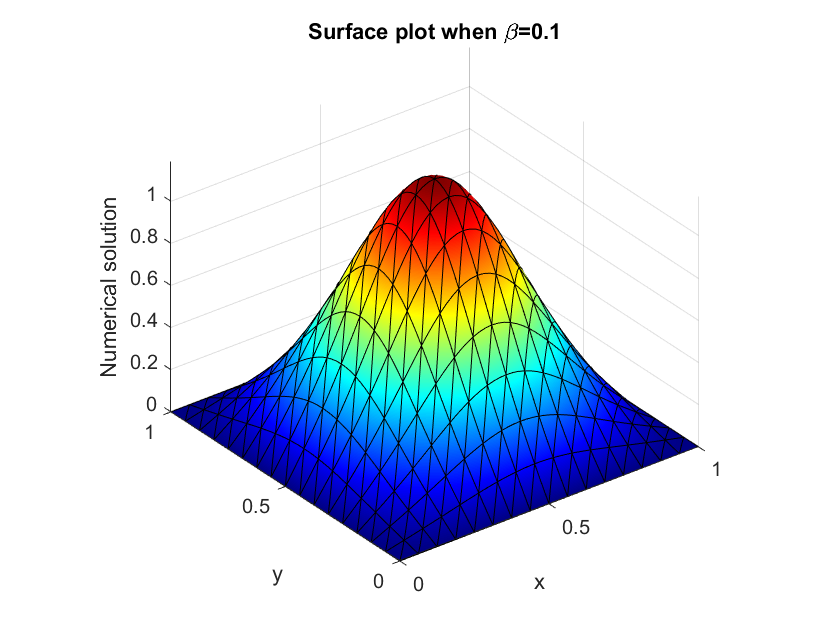}
			\includegraphics[width=0.24\textwidth]{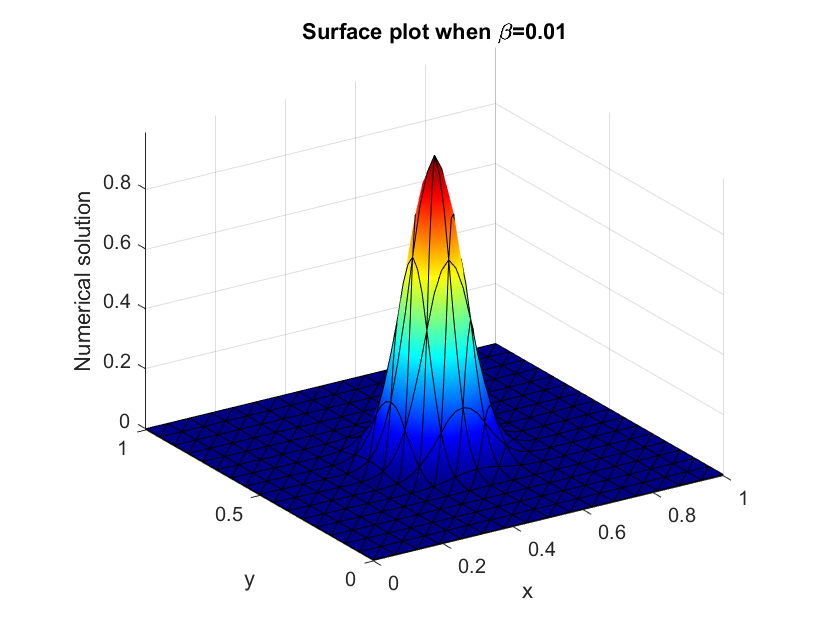}\\
			\includegraphics[width=0.24\textwidth]{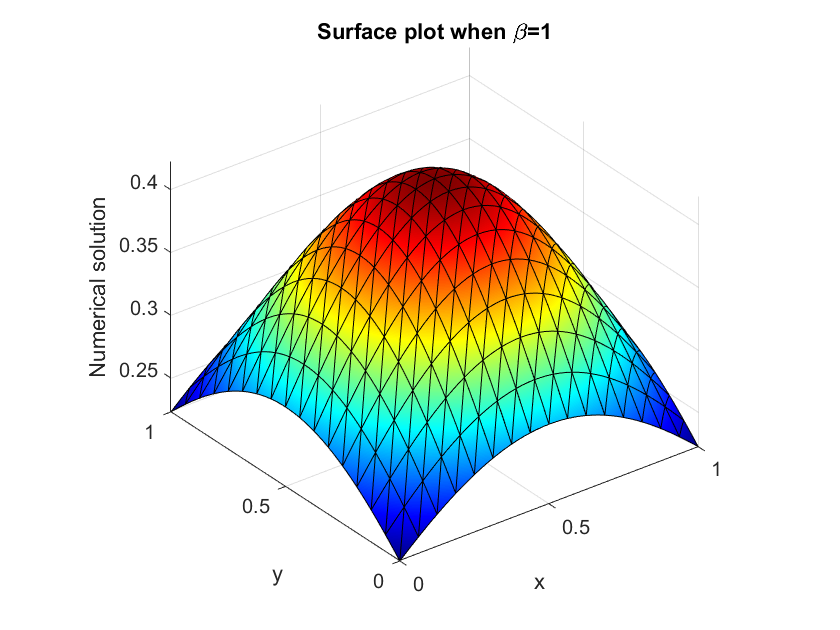}
			\includegraphics[width=0.24\textwidth]{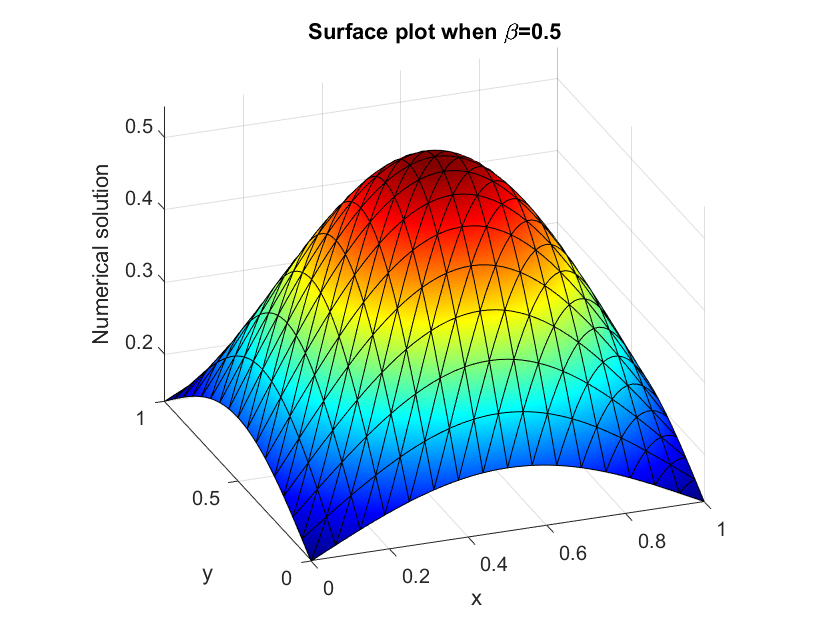}
                        \includegraphics[width=0.24\textwidth]{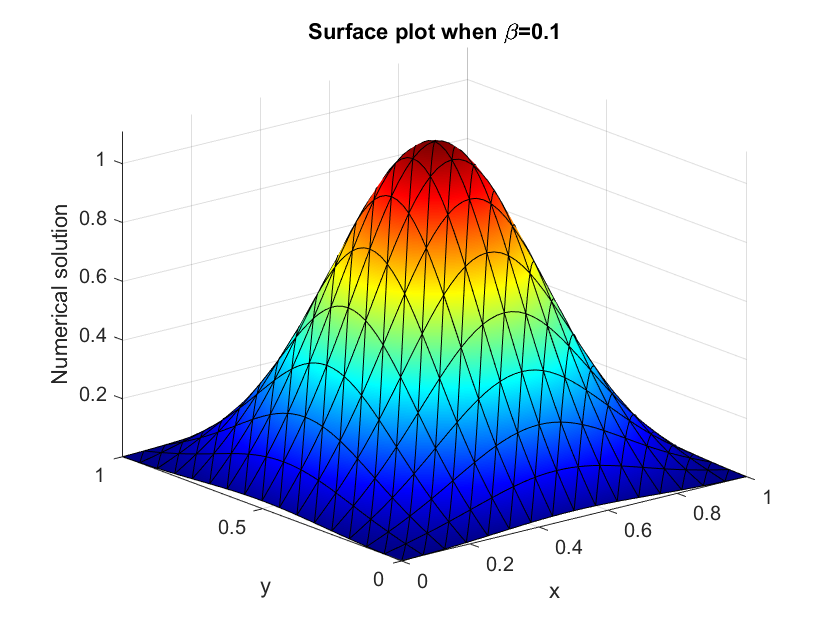}
			\includegraphics[width=0.24\textwidth]{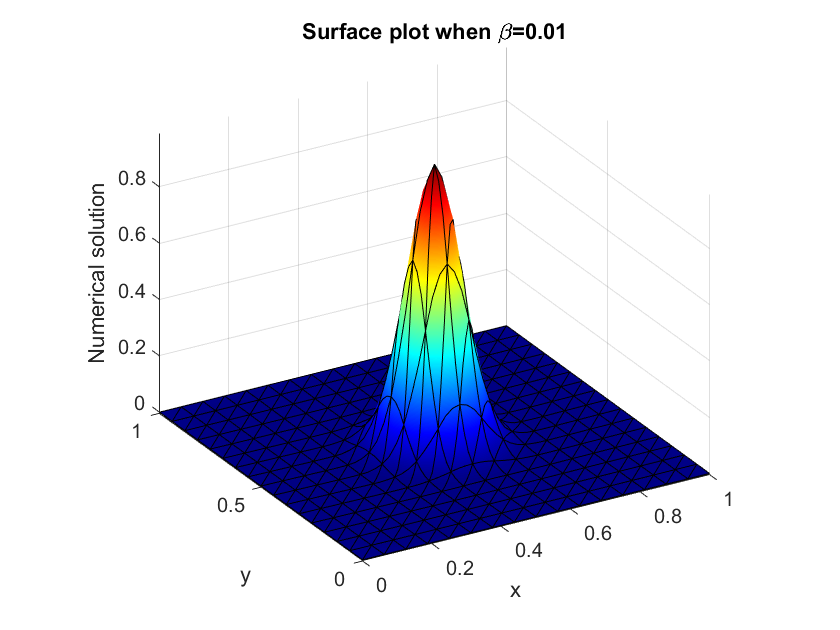}
                          \caption{\Cref{exx2} plots of the numerical solution at time $T=1$ with $h=1/16,k=0.001$ using ${P_2}\times P_1$ elements. Columns moving left to right: $\beta = 1, 0.5, 0.1$, and $0.01$. Rows: Case-I (top) and Case-II (bottom).}\label{figg201}
		\end{figure} 	
	
\end{exm}

\begin{exm}{Homogeneous Rosenau-Burgers in two-dimension)}\label{exx2}
  \rm
		We consider the 2D nonlinear homogeneous Rosenau Burgers-model \eqref{eq1*} having $\alpha=1$ with the following initial condition
		$$u(x,y,0)=exp\bigg(-\frac{(x-0.5)^2}{\beta} -\frac{(y-0.5)^2}{\beta}\bigg), \;\; \;x,y\in\Omega,\nonumber$$ 
                and run two cases (Case I and Case II), where for each case we enforce the $u$ boundary conditions,
                \begin{align*}
                  u(x,y,t) &= u_{I}(x,y,t) = \exp \left(-\frac{(x - 0.5)^2 + (y - 0.5)^2}{\beta}\right), & (x,y,t) &\in \partial \Omega \times (0, T], \\
                  u(x,y,t) &= u_{II}(x,y,t) = e^{-t} u_I(x,y,t), & (x,y,t) &\in \partial \Omega \times (0, T].
                \end{align*}
                We take $\Delta u$ boundary conditions on $(x,y,t) \in \partial \Omega \times (0,T]$ as $\Delta u = \Delta u_I$  and $\Delta u = \Delta u_{II}$ for Case I and II, respectively.
%
%
                The finite element solution with $h=0.001, k=0.01$ and $T=1$ in both cases are shown in \Cref{figg201}.
	
\end{exm}

\begin{exm}{Non-homogeneous 2D Rosenau-Burgers model with homogeneous boundary conditions}\label{exm2}
  \rm
		 Let us consider the following Non-homogeneous, nonlinear 2D Rosenau-Burgers model:
		\begin{eqnarray}
		u_{t}+\Delta^2u_t - \Delta u+\nabla\cdot \vec{\textbf{1}}u+ \vec{\textbf{1}}u.\nabla u =f(x,y,t)\; \ \mbox{in}\;  [0,1]^2\times(0,T],\;T>0,  \label{eq11} 
		\end{eqnarray} 
		with the exact solution $u(x,y,t)=e^{-t}sin(2\pi x)sin(2\pi y).$ Here, vector $\vec{\textbf{1}}$ represents (1,1).
		The initial data and source term $f$ are chosen to match the exact solution.
		
                With $h = 1/16$, $k = 0.001$, and $T = 1$, \cref{fig:ex3} (top row) shows the computed solution on domains $\Omega_1, \Omega_3$, and $\Omega_4$. The middle row in \cref{fig:ex3} shows pointwise error profiles. Finally, the bottom row of \cref{fig:ex3} shows $L^2$, $H^1$, and $H^2$ convergence rates, which match the theoretical rates of \cref{thmerror.,thmBE*}.

                \begin{figure}[htpb]
                  \centering
                  \begin{subfigure}{0.32\linewidth}
                    \centering
                    \includegraphics[width=\textwidth]{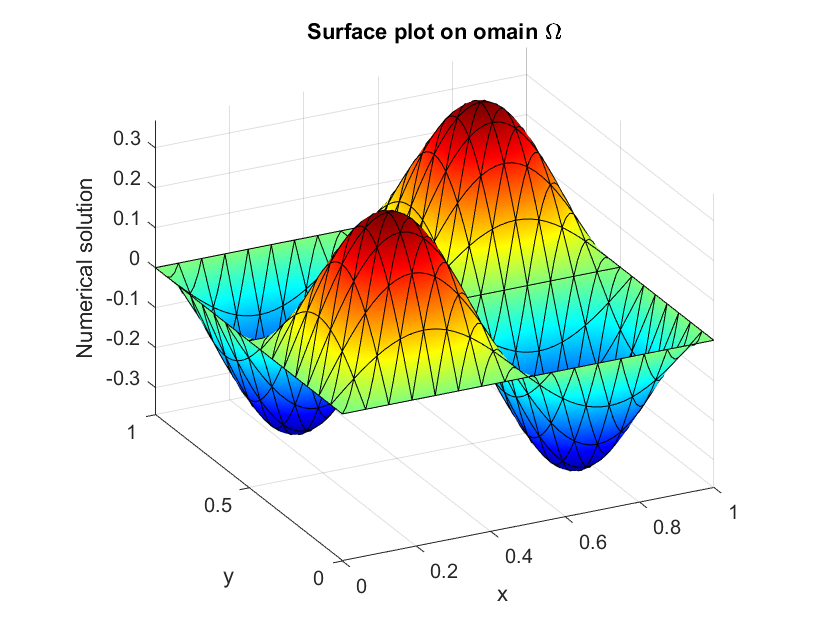}
                  \end{subfigure}
                  \hfill
                  \begin{subfigure}{0.32\linewidth}
                    \centering
                    \includegraphics[width=\textwidth]{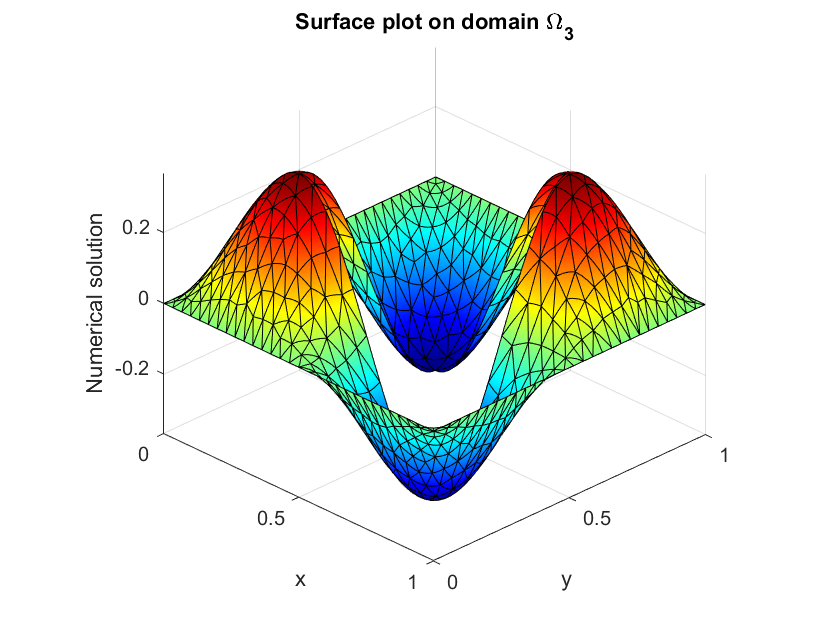}
                  \end{subfigure}
                  \hfill
                  \begin{subfigure}{0.32\linewidth}
                    \centering
                    \includegraphics[width=\textwidth]{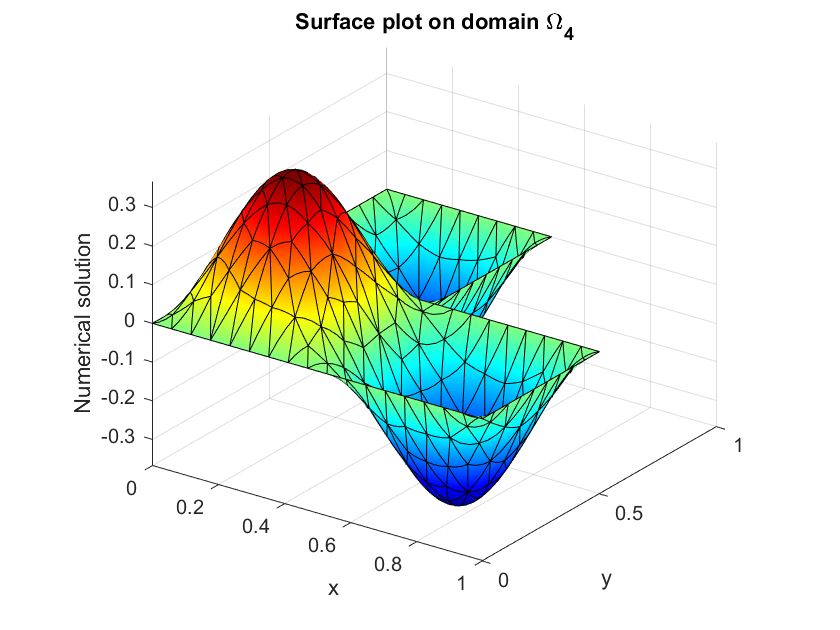}
                  \end{subfigure}
                  
                  \bigskip
                  \begin{subfigure}{0.32\linewidth}
                    \centering
                    \includegraphics[width=\textwidth]{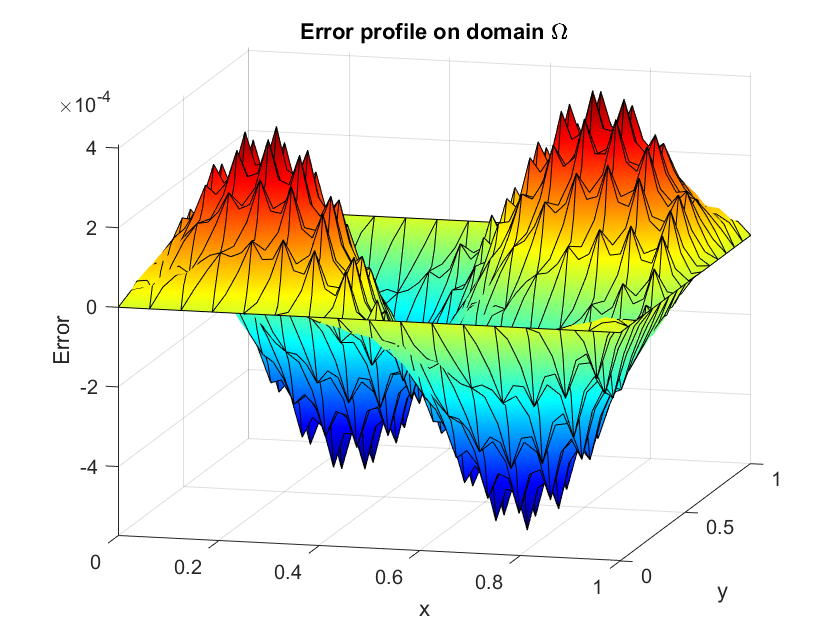}
                  \end{subfigure}
                  \hfill
                  \begin{subfigure}{0.32\linewidth}
                    \centering
                    \includegraphics[width=\textwidth]{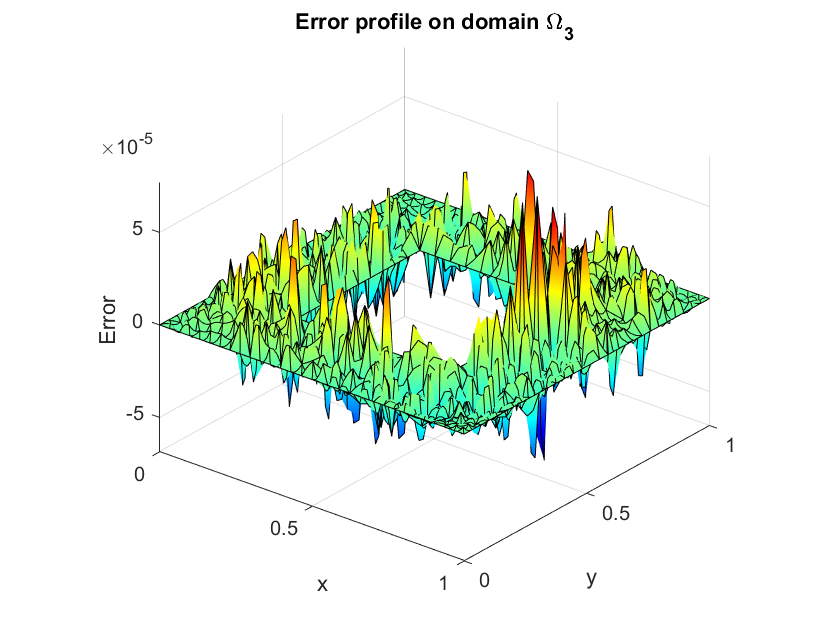}
                  \end{subfigure}
                  \hfill
                  \begin{subfigure}{0.32\linewidth}
                    \centering
                    \includegraphics[width=\textwidth]{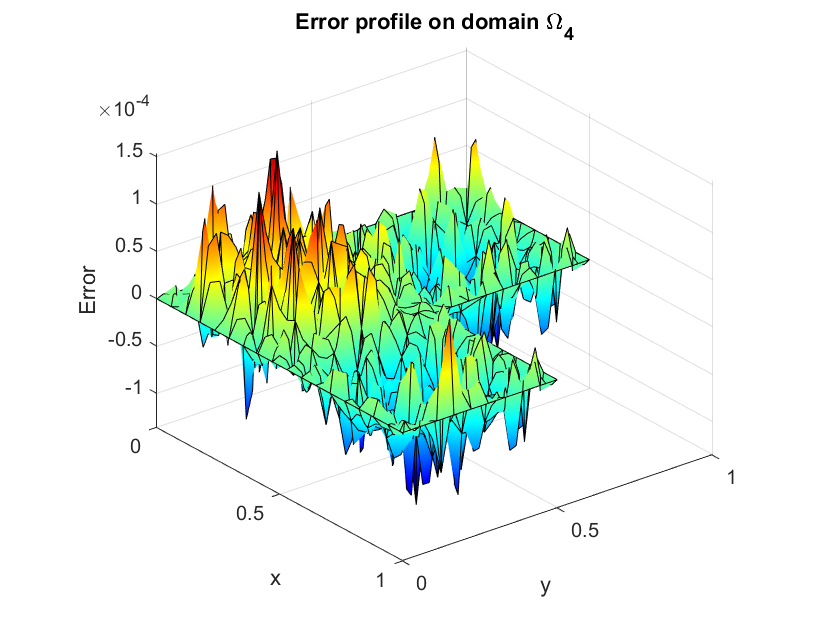}
                  \end{subfigure}
                  
                  \bigskip
                  \begin{subfigure}[c]{0.42\linewidth}
                    \centering
                    \includegraphics[width=\textwidth]{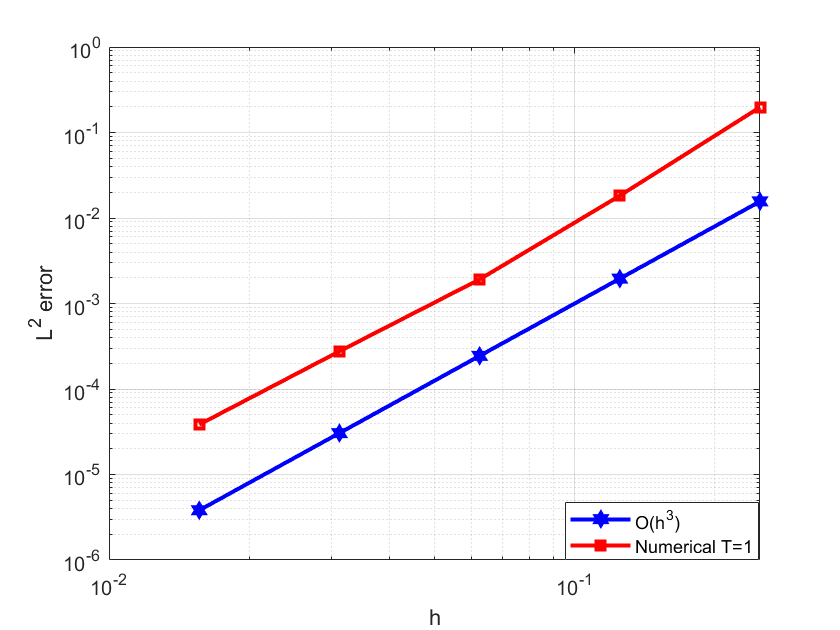}
                  \end{subfigure}
                  \hfill
                  \begin{subtable}{0.56\linewidth}
			\centering
                        \resizebox{\textwidth}{!}{
			\begin{tabular}{c c c a a }
				$ {h} $ &   $\|e(t)\|_{H^{1}}$&$C^{1}-\mbox{order}$ &  $\|e(t)\|_{H^{2}}$&$C^{1}-\mbox{order}$ \\\toprule
				\hline
				$1/32$ & ${6.1970 \times 10^{-3}}$ & $-$& ${1.0052 \times 10^{0}}$ & ${ -}$ \\
				$1/64$ & ${1.5263 \times 10^{-3}}$ & ${2.0215}$& ${5.1285 \times 10^{-1}}$ & ${0.9709}$ \\
				$1/128$ & ${3.7873 \times 10^{-4}}$ & ${2.0108}$& ${2.5901 \times 10^{-1}}$ & ${0.9855}$ \\
				$1/256$ & ${9.4919 \times 10^{-5}}$ & ${ 1.9964}$& ${1.2886 \times 10^{-1}}$ & ${1.0072}$ 
			\end{tabular}
                      }
                  \end{subtable}
                  \caption{\Cref{exm2} with ${P_2}\times P_1$ elements and $k = 0.001$, $T = 1$: Surface (top) and error profiles (middle) of the numerical solutions on domains $\Omega_1$, $\Omega_3$, and $\Omega_4$ with $h=1/16$. Bottom left: $L^2$ errors showing a numerical convergence rate of $h^3$ (red) relative to ideal $h^3$ convergence (blue). Bottom right: Computed orders of convergence in the $H^1$ and $H^2$ norms.}\label{fig:ex3}
                \end{figure}
\end{exm}

\begin{exm}{Non-homogeneous 2D Rosenau-Burgers model with non-homogeneous boundary conditions}\label{exm4}
  \rm 
		Next, we consider the nonlinear, non-homogeneous  2D Rosenau-Burgers model \eqref{eq11} on the domain $\Omega$ corresponding to the exact solution,
		\begin{align}
                  u_e(x,y,t) = exp(2x+2y+2t).\nonumber
		\end{align}
                with $f$ chosen so that $u_e$ is the solution. The boundary conditions for $u$ and $\Delta u$ for $(x,y,t) \in \partial \Omega \times (0, T]$ are set according to $u_e$, as are the initial conditions for $u$ at $t = 0$.
                We use $P_2 \times P_1$ finite elements to compute solutions with $k = 0.001$ up to $T=1$. In \cref{figg105}, top row, we plot numerically computed solutions over domains $\Omega_2, \Omega_5$, and $\Omega_6$ with $h = 1/16$. The middle row of \cref{figg105} shows the corresponding error profiles. The bottom row (left) plots $L^2$ errors, the bottom row (middle) tabulates $H^1$ errors and CPU times for computing the solution, and the bottom row (right) plots $H^2$ errors. All convergence rates are in accordance with the theory provided by \cref{thmerror.,thmBE*}.

			\begin{figure}[htbp]
                                \centering
                                \begin{subfigure}{0.32\linewidth}
                                  \centering
                                  \includegraphics[width=\textwidth]{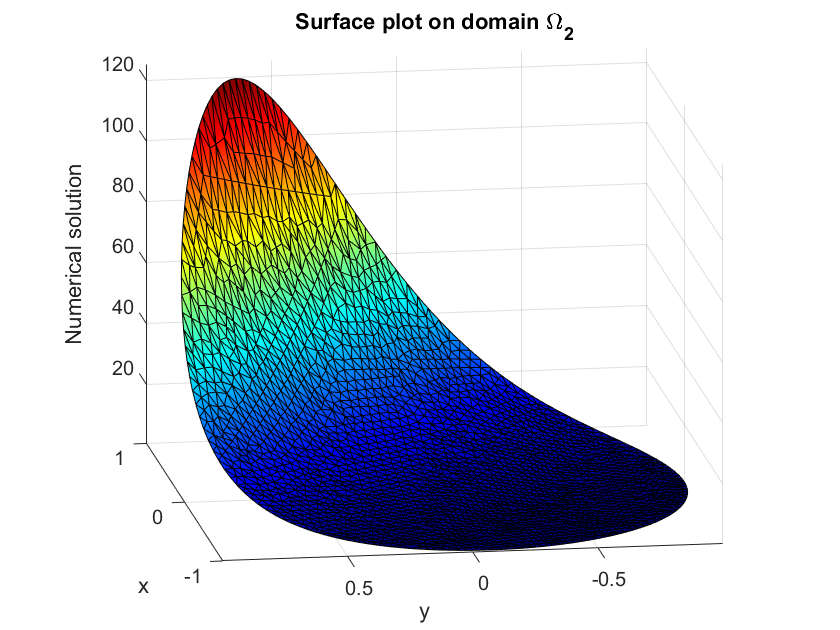}
                                \end{subfigure}
                                \hfill
                                \begin{subfigure}{0.32\linewidth}
                                  \centering
                                  \includegraphics[width=\textwidth]{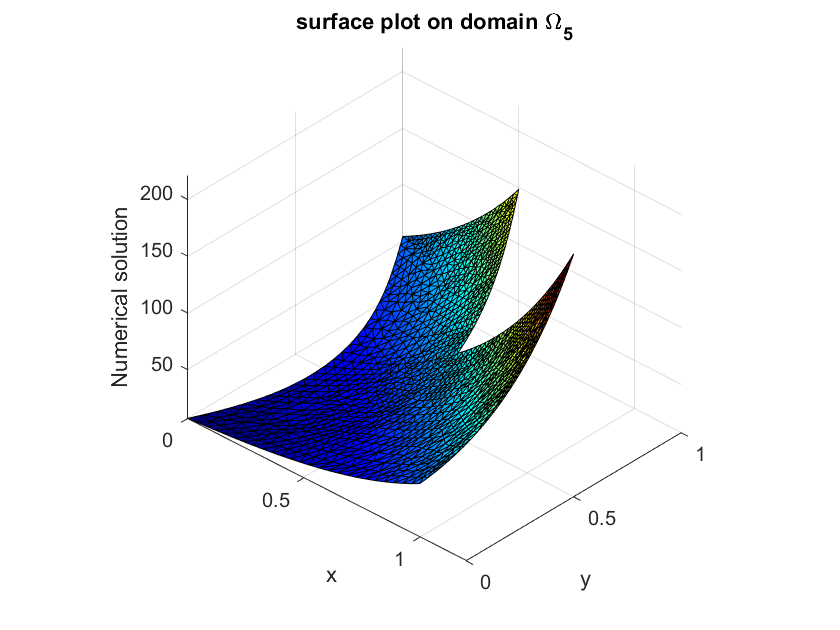}
                                \end{subfigure}
                                \hfill
                                \begin{subfigure}{0.32\linewidth}
                                  \centering
                                  \includegraphics[width=\textwidth]{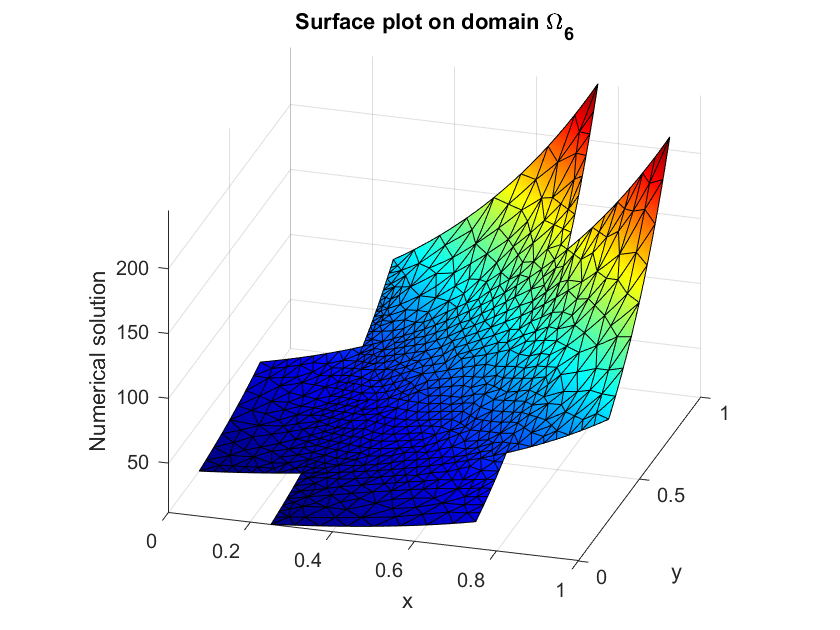}
                                \end{subfigure}
                                
                                \bigskip
                                \begin{subfigure}{0.32\linewidth}
                                  \centering
                                  \includegraphics[width=\textwidth]{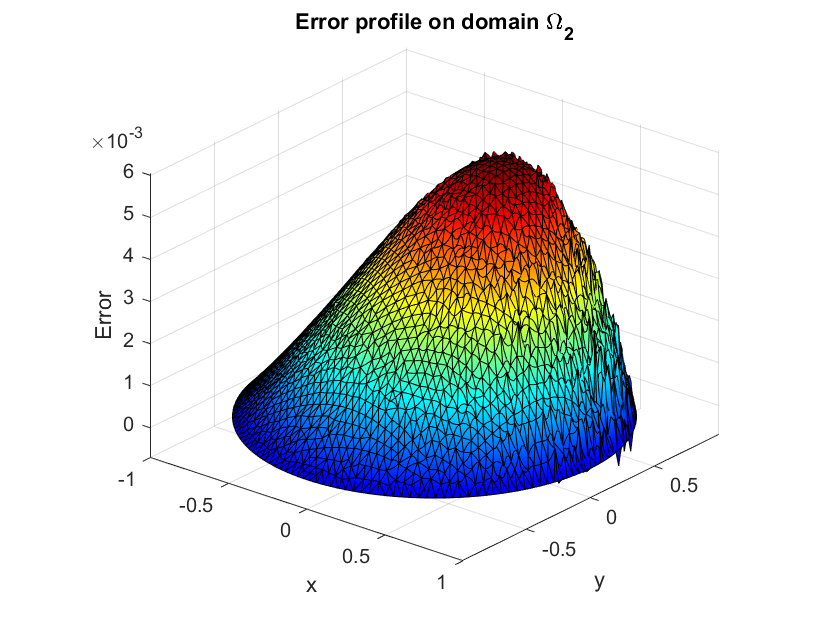}
                                \end{subfigure}
                                \hfill
                                \begin{subfigure}{0.32\linewidth}
                                  \centering
                                  \includegraphics[width=\textwidth]{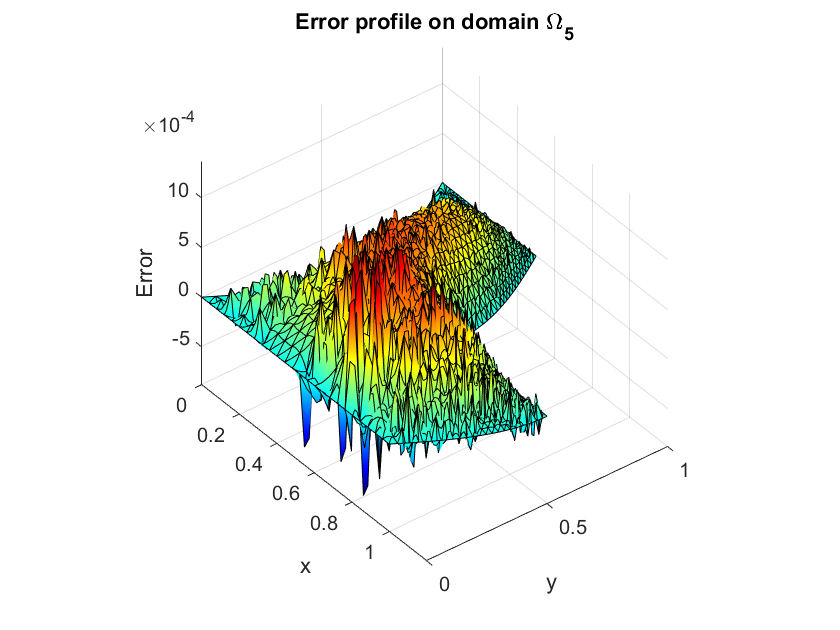}
                                \end{subfigure}
                                \hfill
                                \begin{subfigure}{0.32\linewidth}
                                  \centering
                                  \includegraphics[width=\textwidth]{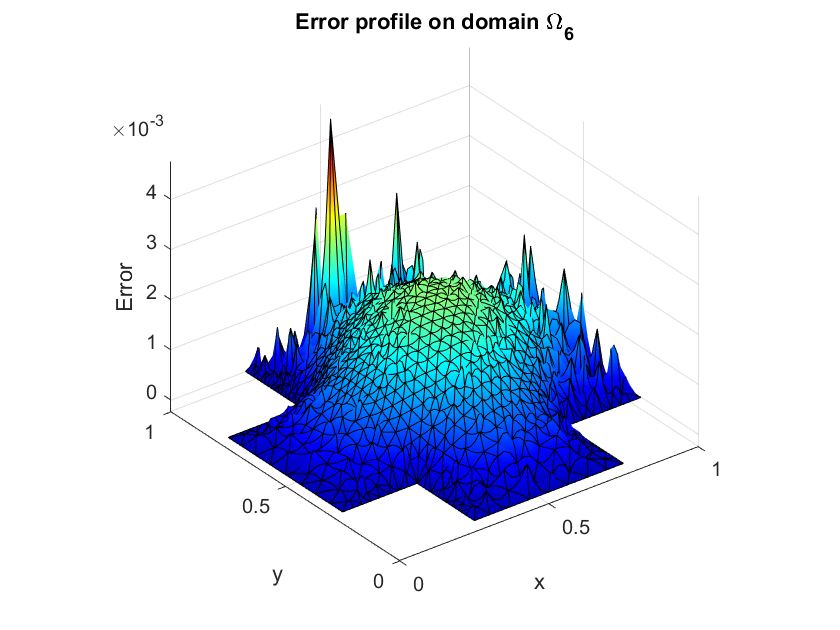}
                                \end{subfigure}
                                \hfill

                                \bigskip
                                \begin{subfigure}[c]{0.32\linewidth}
                                  \centering
                                  \includegraphics[width=\textwidth]{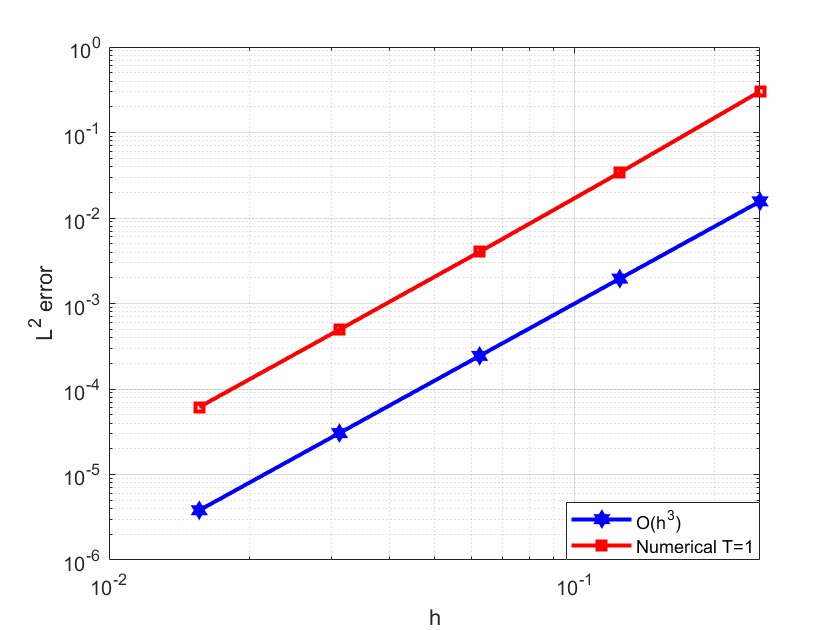}
                                \end{subfigure}
                                \hfill
                                \begin{subtable}{0.32\linewidth}
                                  \centering
                                  \resizebox{\textwidth}{!}{
                                  \begin{tabular}{c c cc   }
                                          $ h$ &   $\|e(t)\|_{H^{1}}$&$C^{1}-\mbox{order}$ &  CPU time \\\toprule
                                          $1/16$ & ${5.8685 \times 10^{-1}}$ & $-$ &$139.72$ \\
                                          $1/32$ & ${1.4140 \times 10^{-1}}$ & $2.0532$ &$163.59$ \\
                                          $1/64$ & ${3.4572 \times 10^{-2}}$ & $2.0321$ &$194.81$ \\
                                          $1/128$ & ${8.5153 \times 10^{-3}}$ & $2.0215$ &$218.12$
                                  \end{tabular}
                                }
                                \end{subtable}
                                \hfill
                                \begin{subfigure}[c]{0.32\linewidth}
                                  \centering
                                  \includegraphics[width=\textwidth]{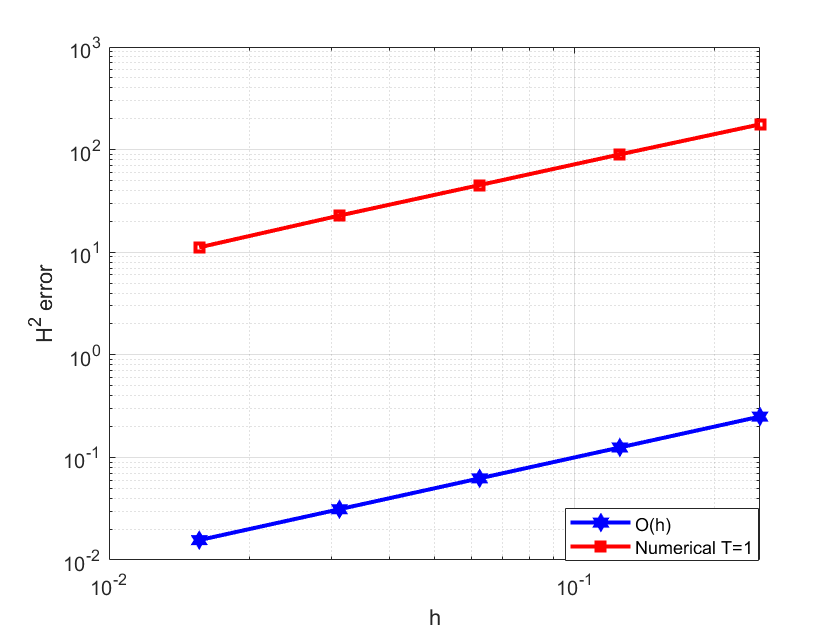}
                                \end{subfigure}
                                \caption{\Cref{exm4} with ${P_2}\times P_1$ elements and $k = 0.001$, $T = 1$: Surface (top) and error profiles (middle) of the numerical solutions on domains $\Omega_2$, $\Omega_5$, and $\Omega_6$ with $h=1/16$. Bottom left: $L^2$ errors showing a numerical convergence rate of $h^3$ (red) relative to ideal $h^3$ convergence (blue). Bottom middle: $H^1$ errors showing a numerical convergence rate of $h^2$ along with a tabulation CPU times required to compute the solutions. Bottom right: $H^2$ errors showing a numerical convergence rate of $h$ (red) relative to ideal $h$ convergence (blue).} \label{figg105}
			\end{figure}

\end{exm}
\section{Conclusion}
We have analyzed a new conformal finite element method using both standard and mixed variational formulations of the nonlinear Rosenau–Burgers-biharmonic model. 
Our existence and uniqueness results show that the mixed approach requires less regularity of the initial condition. 
Our work involves anaylsis for non-homogeneous boundary conditions and error estimates in $H^1$ B\^{o}chner spaces.

\section*{Acknowledgments:} Ankur has been supported by the University Grant Commission India through ID. No.
JUNE18-416131; Ram Jiwari has been supported by the National Board of Higher Mathematics (NBHM), India through grant No. 02011/3/2021NBHM(R.P)/R\&D II/6974;  A. Narayan was partially supported by NSF DMS-1848508.
This material is based upon work supported by both the National Science Foundation under Grant No. DMS-1439786 and the Simons Foundation Institute Grant Award
ID 507536 while A. Narayan was in residence at the Institute for Computational and
Experimental Research in Mathematics in Providence, RI, during the Spring 2020 semester.



\appendix 
\section{Proof of \Cref{lemma:2.2}}\label{sec:lemma2.2-proof}
We seek to prove that $\|u\|_2$ is an equivalent norm to $\|u\| + \|\Delta u\|$. To begin, we require the following well-known inequality.

\begin{lemma}{\bf (Nirenberg inequality \cite{nirenberg1959elliptic}):} 
For $\frac{i}{j}\leq b\leq 1$ with $i\leq j$ and $\frac{1}{p}=\frac{i}{n}+b\big(\frac{1}{r}-\frac{j}{n}\big)+(1-b)\frac{1}{q}$, following estimate holds:
	\begin{equation}
\|D^i u\|_{L^p(\Omega)} \leq C\|D^j u\|_{L^r(\Omega)}^b\|u\|_{L^q(\Omega)}^{1-b},
	\end{equation} 
	where $\Omega$ is the bounded domain in $\mathbb{R}^n$ and the constant C depends upon the domain $\Omega$ as well as $i, n, j, q, r$ and $b$.
\end{lemma}
	
In particular, for $p=q=r=2, i=1, j=2$ and $b=\frac{1}{2}$, we have 
\begin{equation}
\|\nabla u\|\leq \|\Delta u\|^{\frac{1}{2}}\|u\|^{\frac{1}{2}}, \quad \forall u \in  H_0^{2}(\Omega).
\end{equation}
Finally, using Young’s inequality, we obtain
\begin{equation}
\|\nabla u\|^2\leq {\frac{1}{2}}\|\Delta u\|^{2}+{\frac{1}{2}}\|u\|^{2}, \quad \forall u \in  H_0^{2}(\Omega).\;\;\Box\label{nir}
\end{equation}

We can now provide the desired proof.
\begin{proof}[Proof of \Cref{lemma:2.2}]
As $\|u\|_2= \sqrt{\|u\|^2+\|\nabla u\|^2+\|\Delta u\|^2 }$, it is obvious that $\|u\| + \|\Delta u\|\leq C \|u\|_2$.

\noindent Vice versa, exploiting the equation \eqref{nir},
$$ \frac{1}{2}\|u\|_2^2\leq  \frac{1}{2}\|u\|^2+\|\nabla u\|^2+\frac{1}{2}\|\Delta u\|^2\leq \|u\|^2+\|\Delta u\|^2 \leq (\|u\|+\|\Delta u\|)^2. $$
From which we can easily deduce that both are equivalent.
\end{proof}

\section{Nonlinear system obtained in Theorem \ref{exist1}}\label{eq55.}
\begin{proof}
We seek to find a system of ODE in Theorem \ref{exist1}. Equation (\ref{eq30.}) reduces the equation (\ref{eq5.}) as follows:
\begin{subequations}
	\begin{align}
	AX'(t)+BY'(t)+\alpha AY(t)&=G(X(t)),  \label{eq55.a.}\\
	BX(t)&=AY(t),  \label{eq55.a..}\\
	X(0)&=X_0, \label{eq55.b.} 
	\end{align}\label{eq55..}
\end{subequations}
where $$A=[A_{i,j}]_{m\times m},\;A_{i,j}=(\phi_i,\phi_j),\; i,j=1,2 \ldots m,$$
$$B=[B_{i,j}]_{m\times m},\;B_{i,j}=(\nabla\phi_i,\nabla\phi_j),\; i,j=1,2 \ldots m,$$
$$ X=(d_m^1,d_m^2\ldots d_m^m)^T, Y=(b_m^1,b_m^2\ldots b_m^m)^T,$$ $$X_0=(d_m^1(0),d_m^2(0)\ldots d_m^m(0))^T
\;\mbox{and}\;G(X)=(G_j) \;\mbox{with}\; G_j= -( g(u^m), \nabla \phi_j).$$
As matrices $A$ and $B$ are symmetric and positive definite, hence inverse exists for both. Thus, we get the following system of ODE
	\begin{subequations}
	\begin{align}
	X'(t)&=\big[[I+(A^{-1}B)^2]^{-1}A^{-1}\big]\big(G(X(t))-\alpha BX(t)\big),  \label{eq55.a}\\
	X(0)&=X_0. \label{eq55.b} 
	\end{align}
\end{subequations}
\end{proof}

\section{Well-posedness of mixed variational problem}\label{existmixed}
\vspace{-0.2cm}
In this section, we shall prove the existence and uniqueness of the solution to the mixed weak formulation (\ref{eq5}). Before proving the existence and uniqueness, let us derive the stability estimate, which is beneficial to know the long-term behavior of the solution to a time-dependent equation. 
\begin{lemma}[Stability estimate] \label{thm1.}
	Let $(u,p)$ be the solution of \eqref{eq5} and suppose $(u_{0},p_{0}) \in L^2(\Omega)\times L^2(\Omega) $, then the following property holds:
	\begin{equation}
	\|u(t)\|_{1}  \leq  C(\|u_0\|,\|p_0\|)  ,  \;\;t \in(0,T]. \label{st}
	\end{equation} 
	Moreover, there exists a constant $C_1 > 0$ such that
	\begin{equation}
	\|u\|_{L^{\infty}(I; L^{q}(\Omega))}  \leq  C_1(q, \|u_0\|,\|p_0\|),\quad \;\forall\; 1\leq q<\infty.  
	\end{equation}
\end{lemma}
\begin{proof}  Setting $\chi = u$ in the equation \eqref{eq5a} and $\chi' = u,p_t$ in the equation \eqref{eq5b}, we obtained
\begin{equation}	
\dfrac{1}{2} \dfrac{\mathrm{d}}{\mathrm{d}t}\Big[\|u\|^{2} + \| p\|^{2}\Big]  +   \alpha \|\nabla u\|^{2}  = \int_{\Omega} (\nabla\cdot g(u))ud\Omega .\label{eqn5.}
\end{equation}
Using equations \eqref{eqn7} in \eqref{eqn5.} with a positive parameter $\alpha$, we get
$$ \dfrac{\mathrm{d}}{\mathrm{d}t}\Big[\|u\|^{2} + \| p\|^{2}\Big]\leq 0 . $$
We integrate the above inequality w.r.t.  time $t$ to get 
\begin{equation}
\|u\|^{2} + \|p\|^{2} \leq \|u_{0}\|^2+\|p_{0}\|^2.\label{123}
\end{equation}
We can see that using Cauchy Schwartz and Young's inequalities in equation \eqref{eq5b} with $\chi' = u$ gives rise to 
\begin{equation}
\|u(t)\|_{1}\leq \frac{1}{2}\|u\|^{2} + \frac{1}{2}\|p\|^{2}.\label{ce}
\end{equation}
Next, using equations \eqref{123} and \eqref{ce}, we get
\begin{equation}
\|u(t)\|_{1}  \leq  C(\|u_0\|,\|p_0\|)   ,  \;\;\forall \; t \in(0,T].\label{ce...}
\end{equation} 
Finally, we use the fact that $H^{1}(\Omega)$ is embedded in $L^{q}(\Omega)$ for $1\leq q<\infty$  (cf. Sobolev’s Inequality, \cite{quarteroni2009numerical}) in the above lemma to complete the rest of the proof.
\end{proof}

We can now give the proof of \cref{exist1}.
\begin{proof}[Proof of \cref{exist1}]
We shall prove the theorem in the following six parts:
\vspace{-0.2cm}
\begin{enumerate}
	\item \textit{Compression of weak problem:} Let us consider $\{\phi_{j}\}^{\infty}_{j=1}$ to be the orthogonal basis of $H_{0}^{1}(\Omega)$ and suppose  $\digamma^{m}$ = span$\{\phi_{1},\phi_{2},\ldots,\phi_{m}\}$. Now, let us assign the functions $u^m,p^m:[0,T]\to H^1_0(\Omega)$ of the form
	\vspace{-0.2cm}
	\begin{equation}
	u^{m}(t) = \sum_{j=1}^{m} d_m^j(t)\phi_{j}\;\quad\mbox{and}\; \quad p^{m}(t) = \sum_{j=1}^{m} b_m^j(t)\phi_{j} \label{eq30.}
	\end{equation}	
	as a solution of 
	\begin{subequations}
		\begin{align}
		(u_{t}^m(t),\phi_j)+(\nabla p^m_{t}(t),\nabla \phi_j)+\alpha(p^m(t), \phi_j)&=-( g(u^m), \nabla \phi_j),  \label{eq5.a}\\
		(\nabla u^m(t),\nabla \phi_j)&=(p^m(t),  \phi_j), \label{eq5.b} 
		\end{align}\label{eq5.}
	\end{subequations}
	with $u^m(0)=u_{0,m},\;\;\;\mbox{for} \;\phi_j \in \digamma^{m}, j=1,2,\ldots, m.$
	
	$ \mbox{Here,}\;u^{m}(0) = \sum_{j=1}^{m} d_m^j(0)\phi_{j},\; d_m^j(0)=(u_0, \phi_j),\;j=1,2,\ldots m. $ Similarly, we may define $p^{m}(0)$ as well.
	Once we use equation (\ref{eq30.}) in (\ref{eq5.}), we get the nonlinear system of ODE (see Appendix \ref{eq55.} for precisely solved system).
	Now, applying the Picard existence theorem guarantees that there exists a $T^*$ with $0<T^*<T$ such that the above obtained nonlinear system has the local solutions $(u^m(t),p^m(t))$ in the subspace $\digamma^m\times \digamma^m$ for a.e. $t\in [0,T^*].$ Our next goal is to extend the time $T^*$ to $T$.
	\vspace{-0.1cm}
	\item \textit{$H^1-$ estimate:} Multiplying equation (\ref{eq5.a}) by $d_m^j(t)$, equation (\ref{eq5.b}) by $d_m^j(t)$, $b_{m,t}^j(t)$ and then summing over $j=1,2,\ldots, m$ to obtain
	\begin{align}
	\frac{1}{2}\frac{\mathrm{d}}{\mathrm{d}t}\big[\|u^m(t)\|^2+\| p^m(t)\|^2\big] +\alpha \|\nabla u^m(t)\|^2=(\nabla\cdot g(u^m),  u^m(t)).\label{eq32.}
	\end{align}
	
	Integrating equation \eqref{eq32.} from $0$ to $t$ and using equation \eqref{eqn7} yields
	\vspace{-0.2cm}
	\begin{equation}
	\|u^m(t)\|^2+\|p^m(t)\|^2 +\alpha\int_{0}^{t} \|\nabla u^m(s)\|^2ds\leq\|u_0\|^2+\|p_0\|^2.\nonumber
	\end{equation}
	Here, we have used the fact that $\|u^m(0)\| \leq \|u_0\|$ and $\|p^m(0)\|\leq \|p_0\|$. 
	Finally, taking the supremum over $0\leq t\leq T$ to get
	\begin{equation}
	\sup_{0\leq t\leq T} \|u^m(t)\|^2+	\sup_{0\leq t\leq T}\|p^m(t)\|^2+\alpha\int_{0}^{T} \|\nabla u^m(t)\|^2dt\leq C_1.  \label{eq37.}
	\end{equation}
	The above equation confirms that the right hand side is independent of $m$.
	\item \textit{Estimate of time derivative of $u^m$ and $p^m$:}
	From the equation (\ref{eq37.}), it is obvious that
	\vspace{-0.2cm}
	$$\|p^m(t)\|\leq C_1.\vspace{-0.1cm}$$
	Also, we have 
	\begin{eqnarray}
	\|g(u^{m})\|^{2}&=&2\int_{\Omega} \Big(\frac{1}{2}(u^{m})^{2} + u^{m} \Big)^{2} d \Omega\nonumber\\
	&\leq& 2\bigg(\frac{1}{2} \int_{\Omega} (u^{m})^{4} d \Omega +2\int_{\Omega} (u^{m})^{2} d \Omega\bigg)\nonumber\\
	&\leq& C_2(q, \|u_0\|, \|p_0\|). \label{eq39.}
	\end{eqnarray}
	Here, we have used the embedding result $\|u^{m}(t)\|_{L^{q}(\Omega)} \leq C_{3}\|u^{m}(t)\|_{H^{1}(\Omega)}$ for $\;1 \leq q<\infty$ (see, \cite{quarteroni2009numerical}) and stability estimate \eqref{st}. Also, note that $C_2$ is depending on $\|u_0\|$ and $\|p_0\|$.  Consequently,  $g(u^{m})$ is bounded in  $L^{\infty}(I;L^{2}(\Omega)).$
	
	To estimate ${u_t^m}$, multiplying equation (\ref{eq5.a}) by $d_{m,t}^j(t)$,  equation (\ref{eq5.b}) by $b_{m,t}^j(t)$ and then summing over $j=1,2,\ldots m$, we obtain
	\begin{subequations}
		\begin{align}
		(u_{t}^m(t),u_{t}^m(t))+(\nabla p^m_{t}(t),\nabla u_{t}^m(t))&=-\alpha(p^m(t), u_{t}^m(t))-( g(u^m), \nabla u_{t}^m(t)),  \label{eq555.a}\\
		(\nabla u_t^m(t),\nabla p_t^m(t))&=(p_t^m(t),  p_t^m(t)). \label{eq555.b} 
		\end{align}\label{eq555.}
	\end{subequations}
	 \noindent Here, we have used the fact that $p_t^m=-\Delta u_t^m$. Now, first using \eqref{eq555.b} in \eqref{eq555.a} and then applying Cauchy-Schwarz inequality and Young inequality with equation \eqref{ce} in RHS of equation \eqref{eq555.a}, we get a suitable positive constant $C_4$ such that
	$$\|u^{m}_t(t)\|^2\leq C_4(\|g(u^{m})\|^{2}+\|p^{m}(t)\|^2 ).$$
	Integrating the above equation from $0$ to $T$ and using bounds of $p^m(t)$ and $g(u^{m})$, we get a positive constant $C_5$ such that
	\begin{equation}
	\int_{0}^{T}\|u^{m}_t(t)\|^2dt\leq C_5 .\label{eq40.}
	\end{equation}
	To estimate ${p_t^m}$, multiplying equation (\ref{eq5.a}) by $b_{m,t}^j(t)$,  equation (\ref{eq5.b}) by $d_{m,t}^j(t)$ and then summing over $j=1,2,\ldots m$, we obtain
	\begin{subequations}
		\begin{align}
		(u_{t}^m(t),p_{t}^m(t))+(\nabla p^m_{t}(t),\nabla p_{t}^m(t))&=-\alpha(p^m(t), p_{t}^m(t))-( g(u^m), \nabla p_{t}^m(t)),  \label{eq5555.a}\\
		(\nabla u_t^m(t),\nabla u_t^m(t))&=(p_t^m(t),  u_t^m(t)). \label{eq5555.b} 
		\end{align}\label{eq5555.}
	\end{subequations}
	Again, first using \eqref{eq5555.b} in \eqref{eq5555.a} and then applying Cauchy-Schwarz inequality, Young inequality and  Poincar\'{e} inequality, we get a suitable positive constant $C_6$ such that
	$$\|\nabla p^{m}_t(t)\|^2\leq C_6(\|g(u^{m})\|^{2}+\|p^{m}(t)\|^2 ).$$
	Integrating the above equation from $0$ to $T$ and using bounds of $p^m(t)$ and $g(u^{m})$ proved above, we get a positive constant $C_7$ such that
	\begin{equation}
	\int_{0}^{T}\|\nabla p^{m}_t(t)\|^2dt\leq C_7 .\label{eq40..}
	\end{equation}
	\item \textit{Passing to the limit:} Estimates (\ref{eq37.}), (\ref{eq40.}) and (\ref{eq40..}) conclude that the sequence $\{u^m\}_{m=1}^{\infty}$ is uniformly bounded and independent of $m$. Consequently, using the Banach-Alaoglu theorem \cite{chung1994finite,brezis2011functional}, one can extract a subsequence $\{u^{m_k}\}_{k=1}^{\infty}$ of $\{u^m\}_{m=1}^{\infty}$ such that 
	\begin{eqnarray}
	&& u^{m_{k}} \xrightarrow{\textrm{weak}-\ast} u \;\;in\;\; L^{\infty}(I; L^{2}(\Omega)), \nonumber\\&&
	p^{m_{k}} \xrightarrow{\textrm{weak}-\ast} p \;\;in\;\; L^{\infty}(I; L^{2}(\Omega)), \nonumber\\&&
	u^{m_{k}} \xrightarrow{\textrm{weakly}} u \;\;in\;\; L^{2}(I; H_{0}^{1}(\Omega)), \nonumber\\&&
	u^{m_{k}}_t \xrightarrow{\textrm{weakly}} u_t \;\;in\;\; L^{2}(I; L^{2}(\Omega)),\nonumber\\&&
	p^{m_{k}}_t \xrightarrow{\textrm{weakly}} p_t \;\;in\;\; L^{2}(I; H^{1}(\Omega)).\label{eq41.}	
	\end{eqnarray}
	Due to Rellich–Kondrachov theorem, the embedding of $H_{0}^{1}(\Omega)$ is compact in $L^{2}(\Omega)$ \cite{evans2022partial, taylor1996partial}, the Aubin-Lions compactness lemma \cite{simon1986compact} asserts the following strong convergence
	\begin{eqnarray}
	&&u^{m_{k}} {\longrightarrow} \;u \;\;in\;\; L^{2}(I; L^{2}(\Omega)).\label{eq41*.}
	\end{eqnarray}
	Now, we shall try to pass the limit in (\ref{eq5.}) along the subsequence $\{u^{m_{k}}\}_{k=1}^{\infty}$ and $\{p^{m_{k}}\}_{k=1}^{\infty}$. To do the task, let us take a function $\chi, \chi'\in C^1([0,T]; H_0^1(\Omega))$ with the expression $\chi(t) = \sum_{j=1}^{M} d_m^j(t)\phi_{j}$ and $\chi'(t) = \sum_{j=1}^{M} b_m^j(t)\phi_{j}$, where  $\{d_m^j(t)\}_{j=1}^{M}$ and $\{b_m^j(t)\}_{j=1}^{M}$ are smooth functions. Now, choosing $m\geq M$, multiplying equation (\ref{eq5.a}) by $d_m^j(t)$, equation (\ref{eq5.b}) by $b_m^j(t)$, summing from $j=1,2,\ldots, M$ and then integrating it from $0$ to $T$, we obtain 
	\begin{subequations}
		\begin{align}
		&\int_{0}^{T}\bigg[(u_{t}^m(t),\chi(t))+(\nabla p_{t}^m(t),\nabla \chi(t))+\alpha(p^m(t), \chi(t))\bigg]dt\nonumber\\&~~~~~~~~~~~~~~~~~~~~~~~~~~~~~~~~~~ =\int_{0}^{T}\bigg[(\nabla\cdot g(u^m),  \chi(t))\bigg]dt,  \label{eq42.a}\\
		&\int_{0}^{T}\bigg[(\nabla u^m(t),\nabla \chi'(t))\bigg]dt=\int_{0}^{T}\bigg[(p^m(t), \chi'(t))\bigg]dt.\label{eq42.b}
		\end{align}\label{eq42.}
	\end{subequations}
	Let us choose $m=m_k$ and use the convergence  of equations ($\ref{eq41.}$) and ($\ref{eq41*.}$)  in the above equation to get
	\begin{subequations}
		\begin{align}
		&\int_{0}^{T}\bigg[(u_{t}(t),\chi(t))+(\nabla p_{t}(t),\nabla \chi(t))+\alpha(p(t), \chi(t))\bigg]dt\nonumber\\&~~~~~~~~~~~~~~~~~~~~~~~~~~~~~~~~~~ =\int_{0}^{T}\bigg[(\nabla\cdot g(u),  \chi(t))\bigg]dt,  \label{eq43.a}\\
		&\int_{0}^{T}\bigg[(\nabla u(t),\nabla \chi'(t))\bigg]dt=\int_{0}^{T}\bigg[(p(t), \chi'(t))\bigg]dt.\label{eq43.b}
		\end{align}\label{eq43.}
	\end{subequations}
	The density argument asserts that equation (\ref{eq43.}) follows for every $\chi$ and $\chi' \in {L^{2}(I; H_0^1(\Omega))}$. In particular, we get equation \eqref{eq5} for $\chi,\chi' \in H_0^1(\Omega),$ a.e. $t\in[0,T]$.
	\item \textit{Initial data:} Now, we claim that $u(0)=u_0$. From (\ref{eq43.a}), we have
	\begin{align}
	&&\int_{0}^{T}\bigg[(-u(t),\chi_t(t))+(\nabla p_{t}(t),\nabla \chi(t))+\alpha(p(t), \chi(t))\bigg]dt\nonumber\\&& =\int_{0}^{T}\bigg[(\nabla \cdot g(u),  \chi(t))\bigg]dt+(u(0),\chi(0)), \label{eq44.}
	\end{align}
	$\forall$ $\chi\in C^1([0,T]; H_0^1(\Omega))$ with $\chi(T)=0$. Similarly, equation (\ref{eq42.a}) follows
	\begin{align}
	&&\int_{0}^{T}\bigg[-(u^m(t),\chi_{t}(t))+(\nabla p_{t}^m(t),\nabla \chi(t))+\alpha(p^m(t), \chi(t))\bigg]dt\nonumber\\&& =\int_{0}^{T}\bigg[(\nabla\cdot g(u^m),  \chi(t))\bigg]dt+(u^m(0),\chi(0)).  \label{eq45.}
	\end{align}
	Using convergence proved in equations (\ref{eq41.}) and (\ref{eq41*.}) with $m=m_k$, taking the limit along the subsequence $\{u^{m_{k}}\}_{k=1}^{\infty}$, $\{p^{m_{k}}\}_{k=1}^{\infty}$ and using the fact that $u^m(0)\to u_0 \;\mbox{as\;} m \to\infty$ in above leads to
	\begin{align}
	&&\int_{0}^{T}\bigg[-(u(t),\chi_{t}(t))+(\nabla p_{t}(t),\nabla \chi(t))+\alpha(p(t), \chi(t))\bigg]dt\nonumber\\&& =\int_{0}^{T}\bigg[(\nabla\cdot g(u),  \chi(t))\bigg]dt+(u_0,\chi(0)).  \label{eq46.}
	\end{align}
	As $\chi(0)$ is arbitrary, comparing equation (\ref{eq44.}) with (\ref{eq46.}), we can easily deduce that $u(0)=u_0$. 
	\item \textit{Uniqueness:} Let us assume that $(u_1,p_1)$ and $(u_2,p_2)$ are two weak solutions to the equation (\ref{eq1*}). Setting $W:=u_1-u_2$ and $Z:=p_1-p_2$, leads us to 
	\begin{subequations}
		\begin{align}
		(W_{t}(t),\chi)+(\nabla Z_{t}(t),\nabla \chi)+\alpha(Z(t),  \chi)&=-( g(u_1)-g(u_2),  \nabla\chi), \label{eq47.a}\\
		(\nabla W(t),\nabla \chi')&=(Z(t),  \chi'), \label{eq47.b} \\
		W(0)=0, Z(0)&=0. \label{eq47.c}
		\end{align}\label{eq47.}
	\end{subequations}
	Taking $\chi=W$ in (\ref{eq47.a}), $\chi'=W$ in (\ref{eq47.b}) and using Cauchy–Schwarz inequality, Young inequality  and stability estimate for the function $g$, we get
	$$\frac{\mathrm{d}}{\mathrm{d}t}\big[\|W(t)\|^2+\|Z(t)\|^2\big] \leq C\|W(t)\|^2.$$
	Here, we have also used the fact that $(\nabla W(t),\nabla Z_t)=(Z(t),  Z_t(t))$.
	Integrating the above equation w.r.t. time $t$, we have
	$$\|W(t)\|^2+\|Z(t)\|^2\leq\|W(0)\|^{2} +\|Z(0)\|^{2}+ C\int_{0}^{t}\|W(s)\|^{2}  ds.$$
	Finally, the Gronwall’s lemma implies
	$$\|W(t)\|^{2} \leq e^{Ct}\big(\|W(0)\|^{2}+\|Z(0)\|^{2}\big)=0.$$
	Hence, $ u_{1} = u_{2}$ and using (\ref{eq47.b}), we get $ p_{1} = p_{2}.$
\end{enumerate}
\end{proof}

\section{Mixed Formulation Semidiscrete Error Estimates}\label{semimix}
Our main goal in this section is to prove \cref{thmerror.}. Beginning from the notation introduced in \cref{Sec P2}, we introduce the following Ritz projection $\Pi_h$ onto the subspace $S_h$ \cite{thomee2007galerkin}:
\begin{equation}
(\nabla \Pi_{h}u, \nabla \chi)=(\nabla u, \nabla \chi)\;\;\; \forall \chi \in S_h,\;\;\;u \in H_0^1(\Omega).\label{5.3}
\end{equation} 
\begin{equation}
(\nabla \Pi_{h}p, \nabla \chi')=(\nabla p, \nabla \chi')\;\;\; \forall \chi' \in S_h,\;\;\;p \in H_0^1(\Omega).\label{5.4}
\end{equation}

Setting $\chi=\Pi_hu$, we can see that $\|\nabla \Pi_h u\| \leq \|\nabla u\|,\; \forall u \in H_0^1(\Omega)$. Also, one can prove the following property for Ritz projection $\Pi_h$ (Lemma 1.1, Chapter 1, \cite{thomee2007galerkin}):
\begin{equation}
\|\Pi_h u-u\|+h\|\nabla(\Pi_h u-u)\|  \leq Ch^{s}\|u\|_{s}, \label{eq76}
\end{equation}
\begin{equation}
\|\Pi_h p-p\|+h\|\nabla(\Pi_h p-p)\|  \leq Ch^{s}\|p\|_{s}, \label{eq76.}
\end{equation}
for each $u,p\in H^s(\Omega) \cap H^1_0(\Omega),\;\; 1 \leq s \leq r$. 

\noindent Before proceeding further, we derive the following space discrete stability estimate.
\begin{lemma}[Stability estimate] \label{thm1..}
	Let $(u^h,p^h)$ be the solution of \eqref{eq77.} then the following property holds:
	\begin{equation}
	\|u^h(t)\|_{1}  \leq  C (\|u^h_{0}\|, \|p^h_{0}\|),  \;\;t \in(0,T]. \label{st.}
	\end{equation} 
	Moreover, there exists a constant $C_1 > 0$ such that
	\begin{equation}
	\|u^h\|_{L^{\infty}(L^{q}(\Omega))}  \leq  C_1 (q, \|u^h_{0}\|, \|p^h_{0}\|),\quad \;\forall\; 1\leq q<\infty.  
	\end{equation}
\end{lemma}
\begin{proof}
Setting $\chi = u^h$ in the equation \eqref{eq77.a} and $\chi' = u^h,p^h_t$ in the equation \eqref{eq77.b}, we obtained
\begin{equation}	
\dfrac{1}{2} \dfrac{\mathrm{d}}{\mathrm{d}t}\Big[\|u^h\|^{2} + \| p^h\|^{2}\Big]  +   \alpha \|\nabla u^h\|^{2}  = \int_{\Omega} (\nabla\cdot g(u^h))u^hd\Omega .\label{eqn5..}
\end{equation}
Using equations \eqref{eqn7} in \eqref{eqn5..} with positive parameter $\alpha$, we get
$$ \dfrac{\mathrm{d}}{\mathrm{d}t}\Big[\|u^h\|^{2} + \| p^h\|^{2}\Big]\leq 0 . $$
We integrate the above inequality w.r.t.  time $t$ to get 
$$\|u^h(t)\|^{2} + \|p^h(t)\|^{2} \leq \|u^h_{0}\|^2+\|p^h_{0}\|^2.$$
Again, using $\chi' = u^h$ in the equation \eqref{eq77.b} with Cauchy Schwartz and Young's inequalities in the above, we get 
\begin{equation}
\|u^h(t)\|_{1}\leq \frac{1}{2}\|u^h(t)\|^{2} + \frac{1}{2}\|p^h(t)\|^{2}  \leq  C  ,  \;\;t \in(0,T].\label{ce.}
\end{equation} 
Finally, we use the fact that $H^{1}(\Omega)$ is embedded in $L^{q}(\Omega)$ for $1\leq q<\infty$  (cf. Sobolev’s Inequality, \cite{quarteroni2009numerical}) in the above lemma to complete the rest of the proof.
\end{proof}

Finally, we can prove \cref{thmerror.}.
\begin{proof}[Proof of \cref{thmerror.}] First, we write the error in the following form $$u-u^h= (u-\Pi_hu)+(\Pi_hu-u^h)=\rho+\theta,$$
$$p-p^h= (p-\Pi_hp)+(\Pi_hp-p^h)=\eta+\xi.$$
Also, we shall denote $\theta$ for $\theta(t)$ and similarly for $\rho, \eta$ and $\xi$, when no risk of confusion exists.

\noindent Subtracting equation (\ref{eq77.}) from (\ref{eq5}) and introducing Ritz projection $\Pi_h$ with properties \eqref{5.3} and \eqref{5.4},  we get
	\begin{subequations}
\begin{align}
&\big(\theta_t,\chi\big)+\big(\nabla \xi_{t},\nabla \chi\big)+\alpha\big(\xi, \chi\big) =-\big(\rho_t,\chi\big)-\alpha\big(\eta, \chi\big)-\big( g(u)-g(u^h), \nabla \chi\big),\label{eq49.a}\\
&\big(\nabla \theta,\nabla \chi'\big)=\big(\eta, \chi'\big)+\big(\xi, \chi'\big).\label{eq49.b}
\end{align} \label{eq49.}
\end{subequations}
Setting $\chi=\theta$ in \eqref{eq49.a} and using $\chi'=\xi_t, \theta $ in \eqref{eq49.b} to get
\begin{align}
&\frac{1}{2}\frac{\mathrm{d}}{\mathrm{d}t}\bigg[\|\theta(t)\|^2+\|\xi(t)\|^2\bigg]+\alpha\|\nabla \theta\|^2=-\big(\rho_t,\theta\big)-\big(\eta, \xi_t\big)-\big( g(u)-g(u^h), \nabla \theta\big) . \label{eq50.}
\end{align}
Once we use Young inequality and Cauchy–Schwarz inequality in equation \eqref{eq49.b} with $\chi'=\theta$, we get
\begin{equation}
\|\nabla \theta(t)\|^2\leq C\big(\|\theta(t)\|^2+\|\eta(t)\|^2\big)+\|\xi(t)\|^2\big).\label{use}
\end{equation}
Again, using the Cauchy–Schwarz inequality, Young inequality, Lipschitz continuity of function $g$ with the stability estimate and equation \eqref{use} in equation \eqref{eq50.} yields\vspace{-2mm}
\begin{align}
\frac{1}{2}\frac{\mathrm{d}}{\mathrm{d}t}\bigg[\|\theta(t)\|^2+\|\xi(t)\|^2\bigg]&\leq C\big(\|\theta(t)\|^2+\|\eta(t)\|^2+\|\xi(t)\|^2+\|\rho(t)\|^2+\|\rho_t(t)\|^2\big)\nonumber\\
& -\partial_t\big(\eta,\xi\big)+\big(\eta_t,\xi\big).\label{eq50..}
\end{align}
Integrating the above inequality from $0$ to $T$, we get
\begin{align}
&\|\theta(t)\|^2+\|\xi(t)\|^2\leq C\big(\|\eta(t)\|^2+\|\eta(0)\|^2+\|\xi(0)\|^2+\|\theta(0)\|^2\big)\nonumber\\
&~~~~~~~~~~~+\int_{0}^{T} \big(\|\theta(t)\|^2+\|\xi(t)\|^2+\|\eta(t)\|^2+\|\rho(t)\|^2+\|\rho_t(t)\|^2+\|\eta_t(t)\|^2\big).
\label{eq50...}
\end{align}
Using \eqref{eq76} and \eqref{eq76.}, we have\vspace{-2mm}  
\begin{align}
\|\eta(t)\|^2&\leq Ch^{2r}\|p(t)\|_r^2,\nonumber\\\int_{0}^{T}\|\eta(t)\|^2&\leq Ch^{2r}\int_{0}^{T}\|p(t)\|_r^2,\nonumber\\\int_{0}^{T}\|\rho(t)\|^2&\leq Ch^{2r}\int_{0}^{T}\|u(t)\|_r^2,\nonumber\\\int_{0}^{T}\|\eta_t(t)\|^2&\leq Ch^{2r}\int_{0}^{T}\|p_t(t)\|_r^2,\nonumber\\\int_{0}^{T}\|\rho_t(t)\|^2&\leq Ch^{2r}\int_{0}^{T}\|u_t(t)\|_r^2.\nonumber\\\label{eq57.}
\end{align}
Also, considering equation \eqref{eq49.b} at $t=0$ and putting $\chi'=\xi(0)$ while using $\theta(0)=0$, we get the estimate of $\xi(0)$ as follows
\begin{eqnarray}
\|\xi(0)\|^2\leq \frac{1}{2}\|\xi(0)\|^2+\frac{1}{2}\|\eta(0)\|^2.\nonumber
\end{eqnarray}
Hence, we get 
\begin{eqnarray}
\|\xi(0)\|^2\leq Ch^{2r}\|p(0)\|_r^2.\label{111}
\end{eqnarray}
Finally, applying the Gronwall's lemma on equation (\ref{eq50...}) and using the given hypothesis together with estimates (\ref{eq57.}) and  \eqref{111}  give rise to
\begin{eqnarray}
 \|\theta(t)\|^2+\|\xi(t)\|^2\leq Ch^{2r}.\label{eq58.}
\end{eqnarray}
Again, using equations \eqref{eq49.b}, \eqref{eq57.} and \eqref{eq58.}, we get
\begin{eqnarray}
\|\theta(t)\|^2_1\leq Ch^{2r}.\label{eq58..}
\end{eqnarray}
Finally, combining equations \eqref{eq58.} and \eqref{eq58..} with known estimates \eqref{eq76} and \eqref{eq76.} while taking supremum over $0\leq t\leq T$, we get the final estimate \eqref{eq48.}.
\end{proof}

\section{Full Discretization of Mixed Problem $P^2$}\label{fully-mix}
Our main task in this section is to prove \cref{thm:P2-full-wellposed,thmBE*} regarding the discretization \eqref{eq:P2-full}. Firstly, we shall derive the following fully discrete stability estimate:
\begin{lemma} \label{lemma11.}
	Let $(U^m,P^m)$ be the solution of \eqref{eq85.}, then there exists a constant $C$ such that the following estimate holds:
 \begin{equation}
	\|U^{m}\|_{1} \leq C(\|U^{0}\|, \|P^{0}\|),\;\;  m \geq 1. \nonumber
	\end{equation}
 Moreover, there exists a constant $C_1 > 0$ such that
	\begin{equation}
	\|U^m\|_{L^{q}(\Omega)}  \leq  C_1 (q, \|U^{0}\|, \|P^{0}\|),\quad \;\forall\; 1\leq q<\infty.  
	\end{equation}
\end{lemma}

\begin{proof}  Setting $\chi = U^m$ in the equation \eqref{eq85.a} and $\chi' = U^m,\delta_t P^m$ in the equation \eqref{eq85.b}, we obtained
\begin{equation}	
\bigg(\frac{U^{m}-U^{m-1}}{k},U^m\bigg)+\bigg(\frac{P^{m}-P^{m-1}}{k},P^m\bigg)+\alpha(\nabla U^{m}, \nabla U^{m}) =( \nabla\cdot g( U^{m}), U^m) .\label{eqn5...}
\end{equation}
Using Cauchy Schwarz inequality, Young's inequality and equation \eqref{eqn7} in \eqref{eqn5...}  with positive parameter $\alpha$, we get
$$ \delta_t\Big[\|U^m\|^{2} + \| P^m\|^{2}\Big]\leq 0 . $$
Taking the sum from $1$ to $m$, we get 
$$\|U^m\|^{2} + \|P^m\|^{2} \leq \|U^{0}\|^2+\|P^{0}\|^2.$$
Again, using $\chi' = U^m$ in the equation \eqref{eq85.b} with Cauchy Schwartz and Young's inequalities in the above, we get 
\begin{equation}
\|U^m\|_{1}\leq \frac{1}{2}\|U^m\|^{2} + \frac{1}{2}\|P^m\|^{2}  \leq  C (\|U^{0}\|, \|P^{0}\|) ,  \;\;t \in(0,T].\label{ce..}
\end{equation} 
Finally, we use the fact that $H^{1}(\Omega)$ is embedded in $L^{q}(\Omega)$ for $1\leq q<\infty$  (cf. Sobolev’s Inequality, \cite{quarteroni2009numerical}) in the above lemma to complete the rest of the proof.
\end{proof}
We now provide the proof of \cref{thm:P2-full-wellposed}.
\begin{proof}[Proof of \cref{thm:P2-full-wellposed}]  Using the given hypothesis, the initial term $(U^{0},P^0)$ exists vacuously. Moreover, given that $(U^{0},P^0),(U^{1},P^1),\ldots,(U^{m-1},P^{m-1})$ exist. Now, using mathematical induction, we assert that $(U^{m},P^m)$ also exists. To do the task, let us first define an inner product and norm over $S_h\times S_h$:
$$\big(\psi,\psi'\big)_{S_h\times S_h}=\big(\psi_1,\psi'_1\big)_{S_h}+\big(\psi_2,\psi'_2\big)_{S_h},$$ 
$$\|\psi\|_{S_h\times S_h}^2=\|\psi_1\|_{S_h}^2+\|\psi_2\|_{S_h}^2,$$ 
for $\psi=(\psi_1,\psi_2)\;\mbox{and}\; \psi'=(\psi'_1,\psi'_2).$ 

\noindent Also, we set $\Lambda=[\Lambda_1,\Lambda_2]$ on $S_h\times S_h$  by defining maps $\Lambda_1, \Lambda_2:S_h\times S_h \to S_h$ such that
\begin{subequations}	
\begin{align}
&(\Lambda_1(\psi), \chi )_{S_h}= ( \psi_1 ,\chi)_{S_h} +(\nabla \psi_2,\nabla\chi)_{S_h} - ( U^{m-1} ,\chi)_{S_h} - (\nabla P^{m-1},\nabla\chi)_{S_h}\nonumber\\&~~~~~~~~~~~~~~~~~~~~~~~~~ +{{k}}\big[( g(\psi_1) ,\nabla\chi)_{S_h}+\alpha(\nabla \psi_1,\nabla\chi)_{S_h}\big],\;\;\;\; \forall \;\chi \in S_{h},\label{eq123a}\\
&(\Lambda_2(\psi), \chi' )_{S_h}=( \psi_2 ,\chi')_{S_h}-(\nabla \psi_1,\nabla\chi')_{S_h}\label{eq123b}.
\end{align}\label{eq123}
\end{subequations}
\noindent One can easily see that $\Lambda$ is a continuous map. 

\noindent Putting $(\chi,\chi') = (\psi_1,\psi_2)$ in the above equation to obtain
\begin{align}
(\Lambda(\psi), \psi)_{S_h\times S_h}&= (\Lambda_1(\psi), \psi_1)_{S_h}+(\Lambda_2(\psi), \psi_2)_{S_h} \nonumber\\
&= \|\psi\|^{2}_{S_h\times S_h}  - ( U^{m-1} ,\psi_1)_{S_h} - (\nabla P^{m-1},\nabla \psi_1)_{S_h}+ k  \alpha\|\nabla \psi_1\|^{2}_{S_h} ,\nonumber
\end{align}
which implies
\begin{eqnarray}
(\Lambda(\psi), \psi)_{S_h\times S_h}&\geq&\|\psi\|^{2}_{S_h\times S_h} - \frac{1}{2}\|U^{m-1}\|^{2}_{S_h}-\frac{1}{2}\|\psi_1\|^{2}_{S_h} - \frac{1}{4k\alpha}\|\nabla P^{m-1}\|^{2}_{S_h}\nonumber\\
&\geq& \frac{1}{2}\bigg(\|\psi\|^{2}_{S_h\times S_h}-\|U^{m-1}\|^{2}_{S_h}-\frac{1}{2k\alpha}\|\nabla P^{m-1}\|^{2}_{S_h}\bigg).\nonumber
\end{eqnarray}	
For $\|\psi\|_{S_h\times S_h}^2 = \|U^{m-1}\|^{2}_{S_h}+\frac{1}{2k\alpha}\|\nabla P^{m-1}\|^{2}_{S_h}+ C$ (choosing a suitable positive constant $C$ that satisfies the condition of the fixed point theorem), one can easily see that $(\Lambda(\psi), \psi)_{S_h\times S_h} > 0 $. Consequently, Brouwer fixed point theorem guarantees the existence of $\psi^{\star} \in S_h\times S_h$ with $\Lambda(\psi^{\star}) = 0$ such that  $\|\psi^{\star}\|_{S_h\times S_h} \leq \|U^{m-1}\|^{2}_{S_h}+\frac{1}{2k\alpha}\|\nabla P^{m-1}\|^{2}_{S_h}+ C$. Now, selecting  $(U^{m},P^m) = \psi^{\star}$ with $\Lambda(\psi^{\star}) = 0$ satisfies equation (\ref{eq85.}) and thus $(U^{m},P^m)$ exists.

To prove the uniqueness part, we assume that $(U^{m}_{1},P^{m}_{1})$ and $(U^{m}_{2},P^{m}_{2})$ are the two distinct solutions of (\ref{eq85}). Choosing $W^{m}:= U^{m}_{1} - U^{m}_{2}$ and $Q^{m}:= P^{m}_{1} - P^{m}_{2}$, we get
\begin{subequations}
	\begin{align}
	(\delta_t W^m,\chi)+(\nabla \delta_t Q^m,\nabla \chi)+\alpha( Q^{m},  \chi) &=-(  g( U_1^{m})-g( U_2^{m}),  \nabla \chi),                \label{eq86.a}     \\
	(\nabla W^m,\nabla \chi')&=( Q^m,  \chi').                                   \label{eq86.b}                                                                                      \end{align} \label{eq86.}                                        
\end{subequations}
Now, we again use the induction method. Let us first assume that $W^{m-1}=Q^{m-1} = 0$ and we will prove that $W^{m}=Q^{m} = 0.$ For that, setting $\chi =  W^{m}$ in (\ref{eq86.a}) and $\chi' =   \delta_t Q^m, W^{m} $ in (\ref{eq86.b})  yields
\begin{eqnarray}
 \delta_{t} \Big(\|W^m\|^{2} + \|Q^m\|^{2}\Big) 
\leq C\Big(\|W^m\|^{2} + \|\nabla W^m\|^{2}\Big).  \label{eqn28*.}
\end{eqnarray}	
Here, we have  used the stability estimate for function $g$ in RHS such that 
\begin{equation}
\|g(U_{1}^{m})-g(U_{2}^{m})\| \leq C_{1} \| W^{m}\|.
\end{equation}
Setting $\chi' = W^{m} $ in (\ref{eq86.b}) and then using Cauchy Schwarz inequality together with Young's inequality, we get
\begin{equation}
\|\nabla W^m\|^{2}\leq C_2\big(\| W^m\|^{2}+\| Q^m\|^{2}\big).\label{eqn29*.}
\end{equation}
From equations (\ref{eqn28*.}) and (\ref{eqn29*.}), we have 
\begin{eqnarray}
\delta_{t} \Big(\|W^m\|^{2} + \|Q^m\|^{2}\Big) 
\leq C_3\Big(\|W^m\|^{2} + \|Q^m\|^{2}\Big).  \label{eqn28*..}
\end{eqnarray}
Further, utilizing the definition $\delta_{t} \|W^m\|^{2}$ and $\delta_{t} \|Q^m\|^{2}$, we get
$$ \|W^m\|^{2}+\|Q^m\|^{2}\leq \frac{1}{1-C_3k}\big( \|W^{m-1}\|^{2}+\|Q^{m-1}\|^{2}\big).$$
Now, selecting sufficiently small k such that $1-C_3k>0$ and using the  condition $W^{m-1} =Q^{m-1}= 0$, we get $\|W^m\| = \|Q^m\|=0.$ Consequently, $W^m =Q^m= 0$ and hence we obtained the uniqueness of $(U^m,P^m)$.
\end{proof}

Finally, we provide the proof of \cref{thmBE*}:
\begin{proof}[Proof of \cref{thmBE*}] We write the error in the following form $$u(t^m)-U^m= (u(t^m)-\Pi_hu(t^m))+(\Pi_hu(t^m)-U^m)=\rho^m+\theta^m,$$
$$p(t^m)-P^m= (p(t^m)-\Pi_hp(t^m))+(\Pi_hp(t^m)-P^m)=\eta^m+\xi^m.$$
From properties \eqref{5.3} and \eqref{5.4}, we already know $\rho^m$ and $\eta^m$. So, we only need to find the bound of $\theta^m$ and $\xi^m$.

\noindent To do the task, let us subtract equation \eqref{eq85.} from \eqref{eq5} to get 
	\begin{subequations}
	\begin{align}
	&\big(\delta_t\theta^m,\chi\big)+\big(\nabla \delta_t\xi^m,\nabla \chi\big)+\alpha\big(\xi^m, \chi\big)+\alpha\big(\eta^m, \chi\big) =-\big(u_t(t^m)-\delta_t(\Pi_{h}u(t^m)),\chi\big)\nonumber\\
	&~~~~~~~~~-\big(\nabla p_t(t^m)-\nabla \delta_t(\Pi_{h}p(t^m)),\nabla\chi\big)+\big( g(U^m)-g(u(t^m)), \nabla \chi\big),\label{eq49.a.}\\
	&\big(\nabla \theta^m,\nabla \chi'\big)=\big(\eta^m, \chi'\big)+\big(\xi^m, \chi'\big).\label{eq49.b.}
	\end{align} \label{eq49..}
\end{subequations}
 We choose $\chi=\theta^m$ in \eqref{eq49.a.} and $\chi'=\theta^m,\delta_t\xi^m$ \eqref{eq49.b.} to get
\begin{subequations}
	\begin{align}
	&\big(\delta_t\theta^m,\theta^m\big)+\big(\nabla \delta_t\xi^m,\nabla \theta^m\big)+\alpha\big(\nabla \theta^m, \nabla \theta^m\big) =-\big(u_t(t^m)-\delta_t(\Pi_{h}u(t^m)),\theta^m\big)\nonumber\\
	&~~~~~~~~~-\big(\nabla p_t(t^m)-\nabla \delta_t(\Pi_{h}p(t^m)),\nabla \theta^m\big)+\big( g(U^m)-g(u(t^m)), \nabla \theta^m\big),\label{eq49.a..}\\
	&\big(\nabla \theta^m,\nabla \delta_t\xi^m\big)=\delta_{t}\big(\eta^m, \xi^m\big)-\big( \delta_{t}\eta^m,\xi^{m-1}\big)+\big(\xi^m, \delta_t\xi^m\big).\label{eq49.b..}
	\end{align} \label{eq49...}
\end{subequations}
Using equation \eqref{eq49.b..} in \eqref{eq49.a..} with positive coefficient $\alpha$, we have 
\begin{align}
&\big(\delta_t\theta^m,\theta^m\big)+\big(\delta_t\xi^m, \xi^m\big)\leq\big( \delta_{t}\eta^m,\xi^{m-1}\big)-\delta_{t}\big(\eta^m, \xi^m\big)-\big(u_t(t^m)-\delta_t(\Pi_{h}u(t^m)),\theta^m\big)\nonumber\\&~~~~~-\big(\nabla p_t(t^m)-\nabla \delta_t(\Pi_{h}p(t^m)),\nabla \theta^m\big)+\big( g(U^m)-g(u(t^m)), \nabla \theta^m\big).\label{11}
\end{align}
Using the definition of $\delta_t\theta^m$ and $\delta_t\xi^m$ with Cauchy Schwarz inequality and Young's inequality, one can see that 
\begin{equation}
(\delta_t\theta^m,\theta^m)+(\delta_t\xi^m, \xi^m)\geq \frac{1}{2}\delta_t\big(\|\theta^m\|^2+\|\xi^m\|^2\big).\nonumber
\end{equation}
Multiplying both sides by $2k$ and summing from $m=1$ to $J$, above becomes
\begin{equation}
2k\sum_{m=1}^{J}(\delta_t\theta^m,\theta^m)+2k\sum_{m=1}^{J}(\delta_t\xi^m, \xi^m)\geq \|\theta^J\|^2+\|\xi^J\|^2-\|\xi^0\|^2. \label{start}
\end{equation}
Now, we simplify each term of \eqref{11} one by one. Let us  consider the first term after multiplying by $2k$ and summing from $m=1$ to $J$ 
\begin{eqnarray}
2k\sum_{m=1}^{J}(\delta_{t}\eta^m,\xi^{m-1})\leq C\bigg(2k\sum_{m=1}^{J}\|\delta_{t}\eta^m\|^2+2k\sum_{m=1}^{J}\|\xi^{m-1}\|^2\bigg).
\end{eqnarray}
Here,
\begin{align}
\|\delta_{t}\eta^m\|^2&=\|(\delta_{t} p(t^m)-\delta_{t}\Pi_hp(t^m))\|^2\nonumber\\
&=\|\frac{p(t^m)-\Pi_hp(t^m)}{k}-\frac{p(t^{m-1})-\Pi_hp(t^{m-1})}{k}\|^2\nonumber\\&=k^{-2}\|\int_{t^{m-1}}^{t^m}(I-\Pi_h)
{p_t}(\cdot,s)ds\|^2\nonumber\\&\leq k^{-1}\bigg(\int_{t^{m-1}}^{t^m}\|(I-\Pi_h){p_t}(\cdot,s)\|^2ds\bigg)\;\;\; (\mbox{Using}\;\mbox{Holder}\;\mbox{Inequality})\nonumber\\&\leq Ck^{-1}h^{2r}\int_{t^{m-1}}^{t^m}\|{p_t}(\cdot,s)\|_r^2ds.\;\;\; (\mbox{Using}\;\mbox{equation}\; \eqref{eq76.}) \label{aa}
\end{align}
Next, we multiply the second term of \eqref{11} by $2k$ and summing from $m=1$ to $J$ to get
\begin{align}
2k\sum_{m=1}^{J}\delta_{t}(\eta^m, \xi^m)&=C\big[(\eta^J, \xi^J)-(\eta^0, \xi^0)\big]\nonumber \\&\leq C\big(\|\eta^J\|^2+ \|\xi^J\|^2+\|\eta^0\|\|\xi^0\|\big).
\end{align}
From third term of \eqref{11}, we get
\begin{align}
(u_t(t^m)-\delta_t(\Pi_{h}u(t^m)),\theta^m)\leq C\big(\|u_t(t^m)-\delta_t(\Pi_{h}u(t^m)\|^2+\|\theta^m\|^2\big).
\end{align}
Here, 
\begin{align}
\|u_t(t^m)-\delta_t(\Pi_{h}u(t^m))\|^2&\leq C \big(\|u_t(t^m)-\delta_tu(t^m)\|^2+\|\delta_tu(t^m)-\delta_t(\Pi_{h}u(t^m))\|^2\big)\nonumber\\&=Q_1+Q_2. \nonumber
\end{align}
Using Taylor's theorem, we get
\begin{align}
Q_1&=\|{u_t}(t^{m})-\frac{u(t^m)-u(t^{m-1})}{k}\|^2\nonumber\\&=k^{-2}\|\int_{t^{m-1}}^{t^{m}}(s-t^{m-1}){u_{tt}}(\cdot,s)ds\|^2\nonumber\\&\leq \|\int_{t^{m-1}}^{t^{m}}{u_{tt}}(\cdot,s)ds\|^2\nonumber\\&\leq k\int_{t^{m-1}}^{t^{m}}\|{u_{tt}}(\cdot,s)\|^2ds.\;\;\; (\mbox{Using}\;\mbox{Holder}\;\mbox{Inequality}) \label{eq12}
\end{align}
A similar calculation to \eqref{aa} shows that
\begin{equation}
\|\delta_tu(t^m)-\delta_t(\Pi_{h}u(t^m))\|^2 \leq Ck^{-1}h^{2r}\int_{t^{m-1}}^{t^m}\|{u_t}(\cdot,s)\|_r^2ds.
\end{equation}
From the fourth term of \eqref{11}, we have
\begin{align}
(\nabla p_t(t^m)-\nabla\delta_t(\Pi_{h}p(t^m)),\nabla \theta^m)\leq C\big(\|\nabla p_t(t^m)-\nabla\delta_t(\Pi_{h}p(t^m)\|^2+\|\nabla \theta^m\|^2\big).
\end{align}
Here, 
\begin{align}
\|\nabla p_t(t^m)-\nabla\delta_t(\Pi_{h}p(t^m))\|^2&\leq C \big(\|\nabla p_t(t^m)-\nabla\delta_tp(t^m)\|^2\nonumber\\&~~~~~+\|\nabla\delta_tp(t^m)-\nabla\delta_t(\Pi_{h}p(t^m))\|^2\big).\nonumber
\end{align}
A similar calculation done in equations \eqref{aa} and \eqref{eq12} shows that
\begin{align}
\|\nabla p_t(t^m)-\nabla\delta_tp(t^m)\|^2&\leq k\int_{t^{m-1}}^{t^{m}}\|{p_{tt}}(\cdot,s)\|_1^2ds.\\
\|\nabla\delta_tp(t^m)-\nabla\delta_t(\Pi_{h}p(t^m))\|^2&\leq Ck^{-1}h^{2r}\int_{t^{m-1}}^{t^m}\|{p_t}(\cdot,s)\|_{r+1}^2ds.
\end{align} 
Also, choosing $\chi'=\theta^m$ in \eqref{eq49.b.}, we have
\begin{equation}
\|\nabla \theta^m\|^2\leq C (\|\theta^m\|^2+\|\xi^m\|^2+\|\eta^m\|^2).\label{21}
\end{equation}
Next, using the stability estimate for function $g$ in last term of \eqref{11} with above equation provides
\begin{align}
( g(U^m)-g(u(t^m)), \nabla \theta^m) &\leq \| g(U^m)-g(u(t^m))\|\|\nabla \theta^m\| \nonumber\\&\leq C\big(\|\rho^m\|+\|\theta^m\|\big)\|\nabla \theta^m\|  \nonumber\\&\leq C (\|\theta^m\|^2+\|\xi^m\|^2+\|\eta^m\|^2+\|\rho^m\|^2).\label{end}
\end{align}
Finally, multiplying equation \eqref{11} by $2k$, summing from $m=1$ to $J$ and using equations \eqref{start}-\eqref{end}, we have 
\begin{align}
\|\theta^J\|^2+\|\xi^J\|^2&\leq \|\xi^0\|^2+ Ck\sum_{m=1}^{J}(\|\eta^m\|^2+\|\rho^m\|^2)+Ck\sum_{m=0}^{J}(\|\theta^m\|^2+\|\xi^m\|^2)\nonumber\\&+ Ch^{2r}\int_{0}^{t^J}\|{p_t}(\cdot,s)\|_r^2ds +Ch^{2r}\int_{0}^{t^J}\|{u_t}(\cdot,s)\|_r^2ds\nonumber\\&+Ch^{2r}\int_{0}^{t^J}\|{p_t}(\cdot,s)\|_{r+1}^2ds+Ck^2\int_{0}^{t^{J}}\|{u_{tt}}(\cdot,s)\|^2ds\nonumber\\&
+Ck^2\int_{0}^{t^{J}}\|{p_{tt}}(\cdot,s)\|_1^2ds. \label{1111}
\end{align}
For estimating $\xi^0$, choosing $\chi'=\xi^0$ in \eqref{eq49.b.} at initial level $m=0$ follows
$$\big(\nabla \theta^0,\nabla \xi^0\big)=\big(\eta^0, \xi^0\big)+\big(\xi^0,\xi^0\big).$$
Using Cauchy Schwarz inequality and Young's inequality with $\theta^0=0$, we have 
\begin{equation}
 \|\xi^0\|^2\leq \|\eta^0\|^2\leq Ch^{2r}\|p(0)\|_r^2.
\end{equation}
Hence, using the given hypothesis with the above equation in \eqref{1111}, we get
$$\|\theta^J\|^2+\|\xi^J\|^2\leq C(h^{2r}+k^2)+Ck\sum_{m=0}^{J}(\|\theta^m\|^2+\|\xi^m\|^2).$$
which become
$$(1-Ck)\big[\|\theta^J\|^2+\|\xi^J\|^2\big]\leq C(h^{2r}+k^2)+Ck\sum_{m=0}^{J-1}(\|\theta^m\|^2+\|\xi^m\|^2).$$
Selecting $k_1$ with $0<k\leq k_1$ such that $1-Ck>0$ and using discrete Gronwall inequality follows
\begin{equation}
\|\theta^J\|^2+\|\xi^J\|^2\leq C(h^{2r}+k^2), \;\; J=1,2,\ldots N. \label{211}
\end{equation}
Thus equation \eqref{21} becomes
\begin{equation}
\|\theta^J\|^2_1\leq C(h^{2r}+k^2), \;\; J=1,2,\ldots N. \label{222}
\end{equation}
The rest of the proof follows by using equations \eqref{eq76}, \eqref{eq76.}, \eqref{211} and \eqref{222}.
\end{proof}

\bibliographystyle{amsplain}
\bibliography{Ref}
\end{document}